\documentclass[12pt]{amsart}

\usepackage{amsmath,amssymb,amscd,amsxtra,upref,stmaryrd}
\usepackage{latexsym}
\usepackage{mathtools}
\usepackage[dvips]{graphics,epsfig}
\usepackage{enumerate}
\usepackage[colorlinks=true, pdfstartview=FitV, linkcolor=blue, citecolor=blue, urlcolor=blue]{hyperref}
\usepackage{amsthm,amsfonts,color,esint,graphicx, float, mathabx}

\usepackage[margin=1in]{geometry}

\newtheorem{theorem}{Theorem}[section]
\newtheorem{lemma}[theorem]{Lemma}

\newtheorem{corollary}[theorem]{Corollary}

\theoremstyle{remark}

\newtheorem{remark}{Remark}[section]

\newcommand{\com}{\mathbb{C}}
\newcommand{\na}{\mathbb{N}}
\newcommand{\naz}{\mathbb{N}_0}
\newcommand{\re}{\mathbb{R}}
\newcommand{\ent}{\mathbb{Z}}

\newcommand{\rn}{{{\mathbb R}^n}}
\newcommand{\rtn}{\re^{2n}}

\newcommand{\sw}{{\mathcal{S}}(\rn)}
\newcommand{\swp}{{\mathcal{S}'}(\rn)}
\newcommand{\swz}{{\mathcal{S}_0}(\rn)}

\newcommand{\tlw}[4]{\dot F_{#1,#3}^{#2}(#4)} 
\newcommand{\tl}[3]{\dot F_{#1,#3}^{#2}} 
\newcommand{\itlw}[4]{F_{#1,#3}^{#2}(#4)} 
\newcommand{\itl}[3]{F_{#1,#3}^{#2}} 
\newcommand{\besw}[4]{\dot B_{#1,#3}^{#2}(#4)} 
\newcommand{\bes}[3]{\dot B_{#1,#3}^{#2}} 
\newcommand{\ibesw}[4]{B_{#1,#3}^{#2}(#4)} 
\newcommand{\ibes}[3]{B_{#1,#3}^{#2}} 

\newcommand{\lebw}[2]{L^{#1}(#2)} 

\newcommand{\Lp}{L^{p(\cdot)}} 

\newcommand{\pp}{{p(\cdot)}}

\newcommand{\ppo}{{p_1(\cdot)}}
\newcommand{\ppt}{{p_2(\cdot)}}

\newcommand{\mow}[3]{M_{#1}^{#2}(#3)}

\newcommand{\C}{\mathcal{C}}
\newcommand{\M}{\mathcal{M}} 

\newcommand{\Ss}{\mathcal{S}}
\newcommand{\B}{\mathcal{B}}
\def\P{\mathcal P}

\newcommand{\f}{\varphi}

\newcommand{\fhat}{\widehat{f}}
\newcommand{\ghat}{\widehat{g}}
\newcommand{\eixxi}{e^{2\pi i x \cdot \xi}}
\newcommand{\eixeta}{e^{2\pi i x \cdot \eta}}

\newcommand{\eixxe}{e^{2\pi i x \cdot (\xi + \eta)}}

\newcommand{\dx}{\, dx}

\newcommand{\dxi}{\, d\xi}
\newcommand{\deta}{\, d\eta}

\newcommand{\Do}[2]{\Delta^{#1}_{#2}}
\newcommand{\So}[2]{S^{#1}_{#2}}

\newcommand{\abs}[1]{\left\vert #1 \right\vert}
\newcommand{\fr}[2]{{\textstyle \frac{#1}{#2}}}
\newcommand{\norm}[2]{\left\|#1\right\|_{#2}}
\newcommand{\supp}{{\rm{supp}}}

\newcommand{\hcline}{1/p=1/p_1+1/p_2}

\newcommand{\A}{D}
\newcommand{\ex}{d}
\newcommand{\gamt}{\gamma_{p_1,p_2,p,q}^{w_1,w_2, tl}}
\newcommand{\gamb}{\gamma_{p_1,p_2,p,q}^{w_1,w_2,b}}

\DeclareMathOperator*{\essinf}{ess\,inf}
\DeclareMathOperator*{\esssup}{ess\,sup}

\begin{document}

\title[Leibniz-type rules and applications]{Coifman--Meyer multipliers: Leibniz-type rules and applications to scattering of solutions to PDEs.}

\author{Virginia Naibo \and Alexander Thomson}

\address{Virginia Naibo, Department of Mathematics, Kansas State University.
138 Cardwell Hall, 1228 N. 17th Street, Manhattan, KS  66506, USA.}
\email{vnaibo@ksu.edu}

\address{Alexander Thomson, Department of Mathematics, Kansas State University.
138 Cardwell Hall, 1228 N. 17th Street, Manhattan, KS  66506, USA.}
\email{thomson521@ksu.edu}

\thanks{The authors are  partially supported by the NSF under grant DMS 1500381.}



\date{\today}


\begin{abstract}

Leibniz-type rules for Coifman--Meyer multiplier operators are studied in the settings of  Triebel--Lizorkin  and Besov spaces associated to weights in the Muckenhoupt classes. Even in the unweighted case,  improvements  on the currently known estimates are obtained. The flexibility of the methods of proofs allows to prove Leibniz-type rules  in a variety of function spaces that include Triebel--Lizorkin and Besov spaces based on weighted Lebesgue, Lorentz  and Morrey spaces as well as variable-exponent Lebesgue spaces. Applications to scattering properties of solutions to certain systems of partial differential equations involving fractional powers of the Laplacian are presented. 
\end{abstract}

\maketitle


\section{Introduction}

The well-known  fractional Leibniz rules or Kato--Ponce inequalities state that for all $f,g\in \sw$ it holds that
\begin{align}
\norm{D^s(fg)}{L^p} &\lesssim \norm{D^s f}{L^{p_1}} \norm{g}{L^{p_2}} +  \norm{f}{L^{p_1}}   \norm{D^s g}{L^{p_2}},\label{KPH:Lp}\\
\norm{J^s(fg)}{L^p} &\lesssim \norm{J^s f}{L^{p_1}} \norm{g}{L^{p_2}} +  \norm{f}{L^{p_1}}   \norm{J^s g}{L^{p_2}},\label{KPI:Lp}
\end{align}
where $\widehat{D^sf}(\xi)=\abs{\xi}^s\fhat(\xi),$ $\widehat{J^sf}(\xi)=(1+\abs{\xi}^2)^{s/2}\fhat(\xi),$ $s>\tau_p:=n(1/\min(p,1)-1)$ or $s\in 2\na,$ $\hcline,$  $1<p_1,p_2\le \infty$ and $1/2<p\le \infty;$ different pairs of $p_1,$ $p_2$ can be used on the right-hand sides of the inequalities.
Applications of such estimates appear in the study of solutions to partial differential equations such as Euler and Navier-Stokes equations (Kato--Ponce~\cite{MR951744}) and the Korteweg--de~Vries equations (Christ--Weinstein~\cite{MR1124294}, Kenig--Ponce--Vega~\cite{MR1211741}), as well as in investigations of smoothing properties of Schr\"odinger semigroups (Gulisashvili--Kon~\cite{MR1420922}). The reader is referred to the work of  Grafakos--Oh~\cite{MR3200091} and Muscalu--Schlag~\cite{MR3052499} (see also Koezuka--Tomita~\cite{MR3750316}) for the cases corresponding to $1/2< p\le 1$ and to  Bourgain--Li~\cite{MR3263081} (see also Grafakos--Maldonado--Naibo~\cite{MR3189525}) for the case $p=\infty.$

Estimates closely related to \eqref{KPH:Lp} and \eqref{KPI:Lp} where the product $fg$ is replaced by $T_\sigma(f,g)$  have also been  studied; the operator $T_\sigma$ is a bilinear pseudodifferential operator associated to a bilinear symbol $\sigma=\sigma(x,\xi,\eta),$ $x,\xi,\eta\in\rn,$ or a bilinear multiplier $\sigma=\sigma(\xi,\eta),$  $\xi,\eta\in\rn,$ and is given by
$$
T_\sigma(f,g)(x)=\int_{\rtn} \sigma(x,\xi,\eta)\fhat(\xi)\ghat(\eta) \eixxe\dxi\deta.
$$
Such estimates take the form 
\begin{align}
\norm{D^sT_\sigma(f,g)}{Z}\lesssim \norm{D^sf}{X_1} \norm{g}{Y_1} +\norm{f}{X_2}\norm{D^sg}{Y_2},\label{KPH:gral}\\
\norm{J^sT_\sigma(f,g)}{Z}\lesssim \norm{J^sf}{X_1} \norm{g}{Y_1} +\norm{f}{X_2}\norm{J^sg}{Y_2},\label{KPI:gral}
\end{align}
where $Z,$ $X_1,$ $Y_1,$ $X_2$ and $Y_2$ represent various function spaces.
For instance,   Brummer--Naibo~\cite{MR3750234} studied Leibniz-type rules for bilinear pseudodifferential operators  with symbols in certain homogeneous bilinear H\"ormander classes in the setting of function spaces that admit a molecular decomposition and a $\varphi$-transform in the sense of  Frazier--Jawerth~\cite{MR808825,MR1070037};  estimates of the type \eqref{KPH:gral} were proved in Hart--Torres--Wu~\cite{MR3864388} in the context of Lebesgue spaces and mixed Lebesgue spaces for bilinear multiplier operators under  minimal smoothness assumptions on the multipliers; related mapping properties for bilinear pseudodifferential operators with symbols in certain bilinear H\"ormander classes were studied in B\'enyi~\cite{MR1996120} and Naibo--Thomson~\cite{MR3912862} in the scale of Besov spaces, in B\'enyi--Torres~\cite{MR1986065} and B\'enyi--Nahmod--Torres~\cite{MR2250054} in the setting of Sobolev spaces, and in Naibo~\cite{MR3393696}  and Koezuka--Tomita~\cite{MR3750316} in the context of Besov and Triebel--Lizorkin spaces.

Weighted versions of \eqref{KPH:Lp}, \eqref{KPI:Lp}, \eqref{KPH:gral} and  \eqref{KPI:gral} have also been obtained in the context of Lebesgue spaces associated with weights in the Muckenhoupt classes. Indeed, Cruz-Uribe--Naibo~\cite{MR3513582}  proved \eqref{KPH:Lp} and  \eqref{KPI:Lp} for the same range of finite parameters and corresponding weighted spaces $\lebw{p}{w},$ $\lebw{p_1}{w_1}$ and $\lebw{p_2}{w_2}$ with $w_1\in A_{p_1},$ $w_2\in A_{p_2}$ and $w=w_1^{p/p_1}w_2^{p/p_2},$ where $A_{p_1}$ and $A_{p_2}$ denote Muckenhoupt classes. The results  in \cite{MR3513582} also include, among other things,  fractional Leibniz rules in the  setting of weighted Lorentz spaces, Morrey spaces and variable-exponent Lebesgue spaces. On the other hand, Brummer--Naibo~\cite{BrNa2017} proved versions of \eqref{KPH:gral} and  \eqref{KPI:gral} in weighted Lebesgue spaces for bilinear Coifman--Meyer multiplier operators and biparameter Coifman--Meyer multiplier operators.

Current techniques for proving fractional Leibniz rules, as well as some of its extensions to bilinear operators $T_\sigma,$ include the analysis of paraproducts, the mapping properties of bilinear Coifman-Meyer multipliers, uniform estimates for square functions,  vector-valued Fourier multiplier theorems and the use of molecular decompositions. 

The purpose of this article is to use a rather straightforward and quite flexible method for proving  inequalities closely related to \eqref{KPH:gral} and  \eqref{KPI:gral}  in the settings of weighted Triebel-Lizorkin and Besov spaces, with weights in the Muckenhoupt class $A_\infty,$ by means of the function spaces' Nikol'skij representations. The use of such representations in unweighted settings goes back, for instance, to the work of Nikol'skij~\cite{MR0374877}, Meyer~\cite{MR639462}, Bourdaud~\cite{MR673825},  Triebel~\cite{MR3024598}, and  Yamazaki~\cite{MR837335}. We implement the method  for Coifman-Meyer multiplier operators of arbitrary order (see Theorems~\ref{thm:CM:TL:B} and \ref{thm:ICM:TL:B}) and obtain as particular cases improved versions of \eqref{KPH:Lp} and \eqref{KPI:Lp} as well as a number of results that complement and extend the weighted estimates proved in \cite{MR3513582} and \cite{BrNa2017}. As an application, we prove scattering properties for solutions to certain systems of partial differential equations involving fractional powers of the Laplacian  (see Theorems~\ref{thm:scattering} and~\ref{thm:scattering2}). We also show that the flexibility of the methods of proofs allows to obtain estimates of the types \eqref{KPH:gral} and \eqref{KPI:gral} in a variety of function spaces that include Triebel--Lizorkin and Besov spaces based on weighted Lorentz spaces, weighted Morrey spaces and variable-exponent Lebesgue spaces.

The organization of the manuscript is as follows. The statements of the main results and applications are presented in Section~\ref{sec:results}. Notation, definitions and some preliminary results   are given in Section~\ref{sec:prelim}. The proofs of the main results are included in Section~\ref{sec:mainproofs} while the applications are proved in Section~\ref{sec:applications}. In Section~\ref{sec:more}, we illustrate the fact that the strategy applied in the proofs of the results stated in Section~\ref{sec:results} constitutes a unifying approach for obtaining Leibniz-type rules for Coifman--Meyer multiplier operators in a variety of function spaces. Finally, Appendix~\ref{sec:appendix} cotains the main steps in the proofs of Nikol'skij representations for  weighted Triebel--Lizorkin and Besov spaces.

\section{Main results and applications}\label{sec:results}

In this section we present the main results of the manuscript, which will then be complemented in Section~\ref{sec:more}. We refer the reader to Section~\ref{sec:prelim}  for notations and definitions. Briefly,  $\sw$ is the Schwartz class of smooth rapidly decreasing functions defined on $\rn$ and $\swz$ is the subspace of functions in  $\sw$ that have vanishing moments of all orders. Given $w\in A_\infty$ defined on $\rn,$ $0<p,q\le \infty$ and $s\in \re,$ $\tlw{p}{s}{q}{w}$ and $\besw{p}{s}{q}{w}$ refer, respectively, to the weighted homogeneous  Triebel--Lizorkin spaces and Besov spaces on $\rn$ associated to the weight $w;$ the notations $\itlw{p}{s}{q}{w}$ and $\ibesw{p}{s}{q}{w}$ are used for their inhomogeneous counterparts. For $0<p<\infty,$ $H^p(w)$ and $h^p(w)$ denote, respectively, weighted Hardy spaces and weighted local Hardy spaces on $\rn $ associated to $w.$

\subsection{Weighted Leibniz-type rules for Coifman-Meyer multiplier operators}\label{sec:kpcm}

For $m \in \re$, a smooth function $\sigma=\sigma(\xi,\eta),$ $\xi,\eta\in\rn,$  is called a bilinear Coifman-Meyer multiplier of order $m$ if for all multi-indices $\alpha, \beta \in \na_0^n$ there exists a positive constant $C_{\alpha, \beta}$ such that
\begin{equation}\label{eq:CMm}
|\partial_\xi^\alpha \partial_\eta^\beta \sigma(\xi, \eta)| \leq C_{\alpha, \beta} (|\xi|+|\eta|)^{m -(\abs{\alpha}+ \abs{\beta})} \quad \forall (\xi, \eta) \in \re^{2n} \setminus \{(0,0)\}.
\end{equation}

For $w\in A_\infty,$  let $\tau_w=\inf\{\tau\in (1,\infty): w\in A_\tau\};$ given  $0<p,q\le \infty$  denote 
\begin{equation*}
\tau_{p,q}(w) := n \left(\frac{1}{\min(p/\tau_w,q,1)} - 1 \right)  \quad \text{and}\quad \tau_{p}(w):=n\left(\frac{1}{\min(p/\tau_w,1)}-1\right).
\end{equation*} 
If $w\equiv 1,$ in which case $\tau_w=1,$ we just write $\tau_{p,q}$  and $\tau_p,$ respectively. Note  that $\tau_{p,2}(w)=\tau_p(w),$ $\tau_{p,q}(w)\ge \tau_{p,q}$ and $\tau_p(w)\ge \tau_p$ for any $w\in A_\infty.$

Our first main result consists of the following  Leibniz-type rules for Coifman-Meyer multiplier operators in weighted homogeneous Triebel-Lizorkin spaces, weighted homogeneous Besov spaces and weighted Hardy spaces.  As  we will see, improvements of \eqref{KPH:Lp} as well as extensions of known weighted versions of \eqref{KPH:Lp} will be obtained as corollaries of this result (see Remark~\ref{re:improvement}).

\begin{theorem}\label{thm:CM:TL:B}  For $m \in \re,$ let $\sigma(\xi,\eta),$ $\xi,\eta\in\rn,$ be a Coifman-Meyer multiplier of order $m.$ Consider  $0 < p, p_1, p_2  \le \infty$  such that $\hcline$ and  $0 < q \leq \infty;$ let  $w_1,w_2\in A_\infty$ and set $w=w_1^{{p}/{p_1}} w_2^{{p}/{p_2}}.$ 
If $0 < p,p_1,p_2 < \infty$ and  $s > \tau_{p,q}(w),$  it holds that
\begin{equation}\label{KP:CM:TL}
\norm{T_\sigma(f,g)}{\tlw{p}{s}{q}{w}} \lesssim \norm{f}{\tlw{p_1}{s+m}{q}{w_1} } \norm{g}{H^{p_2}(w_2)} +  \norm{f}{H^{p_1}(w_1)}   \norm{g}{\tlw{p_2}{s+m}{q}{w_2} } \quad \forall f, g \in \swz.
\end{equation}
If $0< p, p_1,p_2\leq \infty$ and $s > \tau_p(w)$, it holds that
\begin{equation}\label{KP:CM:B}
\norm{T_\sigma(f,g)}{\besw{p}{s}{q}{w}} \lesssim \norm{f}{\besw{p_1}{s+m}{q}{w_1} } \norm{g}{H^{p_2}(w_2)} +  \norm{f}{H^{p_1}(w_1)}   \norm{g}{\besw{p_2}{s+m}{q}{w_2} } \quad \forall f, g \in \swz,
\end{equation}
where the Hardy spaces $H^{p_1}(w_1)$ and $H^{p_2}(w_2)$ must be replaced by $L^\infty$ if $p_1=\infty$ or $p_2=\infty,$ respectively.

If $w_1=w_2$ then different pairs of $p_1, p_2$ can be used on the right-hand sides of \eqref{KP:CM:TL} and \eqref{KP:CM:B}; moreover, if $w\in A_\infty,$ then 
\begin{equation}\label{KP:CM:TL2}
\norm{T_\sigma(f,g)}{\tlw{p}{s}{q}{w}} \lesssim \norm{f}{\tlw{p}{s+m}{q}{w} } \norm{g}{L^\infty} +  \norm{f}{L^\infty}   \norm{g}{\tlw{p}{s+m}{q}{w}} \quad \forall f, g \in \swz,
\end{equation}
where $0<p<\infty,$ $0<q\le\infty$ and $s>\tau_{p,q}(w).$

\end{theorem}
\begin{remark} If $m\ge 0,$ the estimates above hold true for any $f,g\in \sw$ as long as  $\sw$ is a subspace of the function spaces appearing on the right-hand sides. Such is the case if $1<p_1,p_2<\infty,$ $w_1\in A_{p_1}$ and $w_2\in A_{p_2}$ for  \eqref{KP:CM:TL} and \eqref{KP:CM:B} and if $1<p<\infty$ and $w\in A_p$ for \eqref{KP:CM:TL2}. An analogous   remark applies to the corollaries given below. 
\end{remark}

By means of the lifting property for the weighted homogeneous Triebel-Lizorkin spaces and their relation to weighted Hardy spaces (see Section~\ref{sec:spaces}), the estimates \eqref{KP:CM:TL} and \eqref{KP:CM:TL2} imply the following Leibniz-type rules in the scale of weighted Hardy spaces for operators associated to  Coifman-Meyer multipliers of order zero.

\begin{corollary}\label{coro:KP:CM:Hardy}  Let $\sigma(\xi,\eta),$ $\xi,\eta\in\rn,$ be a Coifman-Meyer multiplier of order $0.$ 
Consider  $0 < p, p_1, p_2  < \infty$  such that $\hcline;$ let  $w_1,w_2\in A_\infty$ and set $w=w_1^{{p}/{p_1}} w_2^{{p}/{p_2}}.$ 
If  $s > \tau_p(w),$ it holds that
\begin{equation}\label{KP:CM:Hardy}
\norm{D^s(T_\sigma(f,g))}{H^p(w)} \lesssim \norm{D^s f}{H^{p_1}(w_1)} \norm{g}{H^{p_2}(w_2)} +  \norm{f}{H^{p_1}(w_1)}   \norm{D^s g}{H^{p_2}(w_2)} \quad \forall f, g \in \swz.
\end{equation}
If $w_1=w_2$ then different pairs of $p_1, p_2$ can be used on the right-hand side of \eqref{KP:CM:Hardy}; moreover, if $w\in A_\infty,$ then 
\begin{equation}\label{Kp:CM:Hardy2}
\norm{D^s(T_\sigma(f,g))}{H^p(w)} \lesssim \norm{D^s f}{H^{p}(w)} \norm{g}{L^\infty} +  \norm{f}{L^\infty}   \norm{D^s g}{H^{p}(w)} \quad \forall f, g \in \swz,
\end{equation}
where $0<p<\infty$ and $s>\tau_{p}(w).$

\end{corollary}

By choosing $\sigma \equiv 1$  so that $T_\sigma(f,g) = fg$, Theorem \ref{thm:CM:TL:B}  implies the following corollary, which, in particular, gives that $\tlw{p}{s}{q}{w}\cap L^\infty$ and $\besw{p}{s}{q}{w}\cap L^\infty$ are quasi-Banach algebras under pointwise multiplication for any $w\in A_\infty.$ 

\begin{corollary}\label{coro:KP:TL:B}  Consider  $0 < p, p_1, p_2  \le \infty$  such that $\hcline$ and  $0 < q \leq \infty;$ let  $w_1,w_2\in A_\infty$ and set $w=w_1^{{p}/{p_1}} w_2^{{p}/{p_2}}.$ 
If $0 < p ,p_1,p_2< \infty$ and  $s > \tau_{p,q}(w),$ it holds that
\begin{equation}\label{KP:TL}
\norm{fg}{\tlw{p}{s}{q}{w}} \lesssim \norm{f}{\tlw{p_1}{s}{q}{w_1}} \norm{g}{H^{p_2}(w_2)} +  \norm{f}{H^{p_1}(w_1)}   \norm{g}{\tlw{p_2}{s}{q}{w_2}} \quad \forall f, g \in \swz.
\end{equation}
If $0 < p, p_1,p_2 \le \infty$ and $s > \tau_p(w)$, it holds that
\begin{equation}\label{KP:B}
\norm{fg}{\besw{p}{s}{q}{w}} \lesssim \norm{f}{\besw{p_1}{s}{q}{w_1}} \norm{g}{H^{p_2}(w_2)} +  \norm{f}{H^{p_1}(w_1)}   \norm{g}{\besw{p_2}{s}{q}{w_2}} \quad \forall f, g \in \swz,
\end{equation}
where the Hardy spaces $H^{p_1}(w_1)$ and $H^{p_2}(w_2)$ must be replaced by $L^\infty$ if $p_1=\infty$ or $p_2=\infty,$ respectively.

If $w_1=w_2$ then different pairs of $p_1, p_2$ can be used on the right-hand sides of \eqref{KP:TL} and \eqref{KP:B}; moreover, if $w\in A_\infty,$ then 
\begin{equation}\label{KP:TL2}
\norm{fg}{\tlw{p}{s}{q}{w}} \lesssim \norm{f}{\tlw{p}{s}{q}{w}} \norm{g}{L^\infty} +  \norm{f}{L^\infty}   \norm{g}{\tlw{p}{s}{q}{w}} \quad \forall f, g \in \swz,
\end{equation}
where $0<p<\infty,$ $0<q\le \infty$ and $s>\tau_{p,q}(w).$
\end{corollary}

In particular, setting $q=2$ for \eqref{KP:TL} and \eqref{KP:TL2} (or setting $\sigma\equiv 1$ in Corollary~\ref{coro:KP:CM:Hardy}), we obtain:

\begin{corollary}\label{coro:KP:Hardy} 
Consider  $0 < p, p_1, p_2  < \infty$  such that $\hcline;$ let  $w_1,w_2\in A_\infty$ and set $w=w_1^{{p}/{p_1}} w_2^{{p}/{p_2}}.$ 
If  $s > \tau_{p}(w),$ it holds that
\begin{equation}\label{KP:Hardy}
\norm{D^s(fg)}{H^p(w)} \lesssim \norm{D^s f}{H^{p_1}(w_1)} \norm{g}{H^{p_2}(w_2)} +  \norm{f}{H^{p_1}(w_1)}   \norm{D^s g}{H^{p_2}(w_2)} \quad \forall f, g \in \swz.
\end{equation}
If $w_1=w_2$ then different pairs of $p_1, p_2$ can be used on the right-hand side of \eqref{KP:Hardy}; moreover, if $w\in A_\infty,$ then 
\begin{equation*}
\norm{D^s(fg)}{H^p(w)} \lesssim \norm{D^s f}{H^{p}(w)} \norm{g}{L^\infty} +  \norm{f}{L^\infty}   \norm{D^s g}{H^{p}(w)} \quad \forall f, g \in \swz,
\end{equation*}
where $0<p<\infty$ and $s>\tau_{p}(w).$
\end{corollary}

\begin{remark}\label{re:BN}
The estimates in Corollary~\ref{coro:KP:CM:Hardy} are related to some of those  in \cite[Theorem 1.1]{BrNa2017}, where it was proved, using different methods,  that if $\sigma$ is a Coifman-Meyer multiplier of order 0, $1<p_1,p_2\le \infty,$ $\frac{1}{2}<p<\infty,$ $\hcline,$  $w_1\in A_{p_1},$ $w_2\in A_{p_2},$ $w=w_1^{{p}/{p_1}} w_2^{{p}/{p_2}}$ and $s>\tau_p,$ then for all $f,g\in \sw$ it holds that
\begin{equation}\label{KP:CM:Lebesgue}
\norm{D^s(T_\sigma(f,g))}{L^p(w)} \lesssim \norm{D^s f}{L^{p_1}(w_1)} \norm{g}{L^{p_2}(w_2)} +  \norm{f}{L^{p_1}(w_1)}   \norm{D^s g}{L^{p_2}(w_2)}. 
\end{equation}
Moreover, if $w_1=w_2$ then different pairs of $p_1, p_2$ can be used on the right-hand side of \eqref{KP:CM:Lebesgue}

Corollary~\ref{coro:KP:CM:Hardy} and  \cite[Theorem 1.1]{BrNa2017} have some overlap but each of them gives a different set  of estimates:

\begin{enumerate}[$\bullet$]
\item The estimate \eqref{KP:CM:Hardy} allows for  $0<p,p_1,p_2<\infty,$ for any $w_1,w_2\in A_\infty$ and for the norm in $H^p(w)$ on its left-hand side as long as $s>\tau_{p}(w).$ On the other hand,  \eqref{KP:CM:Lebesgue} requires $1<p_1,p_2\le \infty,$  $w_1\in A_{p_1},$ $w_2\in A_{p_2},$  the norm in $L^p(w)$ on its left-hand side  and $s>\tau_p.$ Therefore, recalling that $\tau_p\le\tau_{p}(w),$ when compared to \eqref{KP:CM:Lebesgue}, the estimate \eqref{KP:CM:Hardy} is less restrictive regarding the ranges for $p,p_1,p_2$ and the classes of weights, but more restrictive in terms of the range for the regularity $s.$

\item In particular, \eqref{KP:CM:Hardy}  implies \eqref{KP:CM:Lebesgue} for $s>\tau_{p}(w),$   $1/2<p<\infty,$ $1<p_1,p_2<\infty$ such that $\hcline,$  $w_1\in A_{p_1}$ and $w_2\in A_{p_2}.$   However, if  $\tau_p<\tau_{p}(w),$ then \eqref{KP:CM:Hardy} does not give \eqref{KP:CM:Lebesgue} for $\tau_p<s\le  \tau_{p}(w),$ while  \eqref{KP:CM:Lebesgue} holds for $s>\tau_p.$ The following are examples of  weights $w_1$ and $w_2$ for which the corresponding weight $w$ satisfies $\tau_p<\tau_{p}(w)$:  Let  $1<p_1\le p_2<\infty$ and $w_1(x)=w_2(x)=w(x)=\abs{x}^a$ with $n(r-1)<a<n(p_1-1)$ for some $1<r<p_1.$  Then    $\abs{x}^a\in A_{p_1}\subset A_{p_2}$ and $\abs{x}^a\notin A_r;$  the latter  gives $\tau_w>1,$ which implies  $\tau_p<\tau_{p}(w)$ if $p<\tau_{w}.$

\item For $1<p<\infty,$ $w\in A_p$ and $s>\tau_p,$ \eqref{KP:CM:Lebesgue} gives  the estimate \eqref{Kp:CM:Hardy2}   as well as the endpoint estimate
 \begin{equation*}
\norm{D^s(T_\sigma(f,g))}{L^p(w)} \lesssim \norm{D^s f}{L^\infty} \norm{g}{L^{p}(w)} +  \norm{f}{L^{p}(w)}   \norm{D^s g}{L^\infty}.
\end{equation*}
 On the other hand, \eqref{Kp:CM:Hardy2} allows  for  $0<p<\infty$ and  $w\in A_\infty$  as long as $s>\tau_{p}(w).$ 
\end{enumerate}

\end{remark}

\begin{remark}\label{re:improvement}
Notice that when $w_1=w_2\equiv 1,$ the inequality \eqref{KP:Hardy} extends and improves \eqref{KPH:Lp} by allowing $0<p,p_1,p_2<\infty.$  In particular, if $1<p_1,p_2<\infty,$  \eqref{KP:Hardy} gives \eqref{KPH:Lp} with the larger quantity $\norm{D^s(fg)}{H^p}$ on the left-hand side. Moreover, Corollary~\ref{coro:KP:Hardy} complements some of the estimates obtained through different methods in \cite[Theorem 1.1]{MR3513582} in the same manner  Corollary~\ref{coro:KP:CM:Hardy} complements  \cite[Theorem 1.1]{BrNa2017} as explained in Remark~\ref{re:BN}; as in that case, 
 Corollary~\ref{coro:KP:Hardy} and \cite[Theorem 1.1]{MR3513582} have some estimates in common but each of them gives a different set of results.
 \end{remark}

\subsection{Weighted Leibniz-type rules for inhomogeneous Coifman--Meyer multiplier operators}\label{sec:inhCM} In this section we consider bilinear multiplier operators $T_\sigma$ where $\sigma=\sigma(\xi,\eta)$ satisfies the estimates \eqref{eq:CMm} with $\abs{\xi}+\abs{\eta}$ replaced with $1+\abs{\xi}+\abs{\eta};$ such multipliers are better suited for the setting of inhomogeneous spaces and we will refer to them as inhomogeneous Coifman--Meyer multipliers. As it will  become apparent from the proofs, an approach akin to the one used in the homogeneous setting leads to results for inhomogeneous Coifman--Meyer multiplier operators, in the spirit of those stated in Section~\ref{sec:kpcm}, in the context of weighted inhomogeneous Triebel--Lizorkin spaces, weighted inhomogeneous Besov spaces and weighted local Hardy spaces. Specifically, we have:

\begin{theorem}\label{thm:ICM:TL:B}  For $m \in \re,$ let $\sigma(\xi,\eta),$ $\xi,\eta\in\rn,$ be an inhomogeneous Coifman-Meyer multiplier of order $m.$ Consider  $0 < p, p_1, p_2  \le \infty$  such that $\hcline$ and  $0 < q \leq \infty;$ let  $w_1,w_2\in A_\infty$ and set $w=w_1^{{p}/{p_1}} w_2^{{p}/{p_2}}.$ 
If $0 < p,p_1,p_2 < \infty$ and  $s > \tau_{p,q}(w),$  it holds that
\begin{equation}\label{KP:ICM:TL}
\norm{T_\sigma(f,g)}{\itlw{p}{s}{q}{w}} \lesssim \norm{f}{\itlw{p_1}{s+m}{q}{w_1}} \norm{g}{h^{p_2}(w_2)} +  \norm{f}{h^{p_1}(w_1)}   \norm{g}{\itlw{p_2}{s+m}{q}{w_2}} \quad \forall f, g \in \sw.
\end{equation}
If $0< p, p_1,p_2\leq \infty$ and $s > \tau_p(w)$, it holds that
\begin{equation}\label{KP:ICM:B}
\norm{T_\sigma(f,g)}{\ibesw{p}{s}{q}{w}} \lesssim \norm{f}{\ibesw{p_1}{s+m}{q}{w_1} } \norm{g}{h^{p_2}(w_2)} +  \norm{f}{h^{p_1}(w_1)}   \norm{g}{\ibesw{p_2}{s+m}{q}{w_2} } \quad \forall f, g \in \sw,
\end{equation}
where the local Hardy spaces $h^{p_1}(w_1)$ and $h^{p_2}(w_2)$ must be replaced by $L^\infty$ if $p_1=\infty$ or $p_2=\infty,$ respectively.

If $w_1=w_2$ then different pairs of $p_1, p_2$ can be used on the right-hand sides of \eqref{KP:ICM:TL} and \eqref{KP:ICM:B}; moreover, if $w\in A_\infty,$ then 
\begin{equation*}
\norm{T_\sigma(f,g)}{\itlw{p}{s}{q}{w}} \lesssim \norm{f}{\itlw{p}{s+m}{q}{w} } \norm{g}{L^\infty} +  \norm{f}{L^\infty}   \norm{g}{\itlw{p}{s+m}{q}{w}} \quad \forall f, g \in \sw,
\end{equation*}
where $0<p<\infty,$ $0<q\le\infty$ and $s>\tau_{p,q}(w).$
\end{theorem}

 Corollaries of Theorem~\ref{thm:ICM:TL:B} analogous to those in Section~\ref{sec:kpcm} follow with the operator $D^s$ replaced by the operator $J^s.$   For instance, we have:
 \begin{corollary}\label{coro:KP:CM:Hardyloc}  Let $\sigma(\xi,\eta),$ $\xi,\eta\in\rn,$ be an inhomogeneous Coifman-Meyer multiplier of order $0.$ 
Consider  $0 < p, p_1, p_2  < \infty$  such that $\hcline;$ let  $w_1,w_2\in A_\infty$ and set $w=w_1^{{p}/{p_1}} w_2^{{p}/{p_2}}.$ 
If  $s > \tau_p(w),$ it holds that
\begin{equation*}
\norm{J^s(T_\sigma(f,g))}{h^p(w)} \lesssim \norm{J^s f}{h^{p_1}(w_1)} \norm{g}{h^{p_2}(w_2)} +  \norm{f}{h^{p_1}(w_1)}   \norm{J^s g}{h^{p_2}(w_2)} \quad \forall f, g \in \sw.
\end{equation*}
If $w_1=w_2$ then different pairs of $p_1, p_2$ can be used on the right-hand side of \eqref{KP:CM:Hardy}; moreover, if $w\in A_\infty,$ then 
\begin{equation*}
\norm{J^s(T_\sigma(f,g))}{h^p(w)} \lesssim \norm{J^s f}{h^{p}(w)} \norm{g}{L^\infty} +  \norm{f}{L^\infty}   \norm{J^s g}{h^{p}(w)} \quad \forall f, g \in \sw,
\end{equation*}
where $0<p<\infty$ and $s>\tau_{p}(w).$
\end{corollary}

Corollary~\ref{coro:KP:CM:Hardyloc} complements some of the estimates obtained in \cite[Theorem 1.1]{BrNa2017} for $J^s$ in an analogous way to that described in Remark~\ref{re:BN}. Moreover, Corollary~\ref{coro:KP:CM:Hardyloc} applied to the case $\sigma\equiv 1$ gives in particular 
\begin{equation}\label{eq:kplochardy}
\norm{J^s(fg)}{h^p(w)} \lesssim \norm{J^s f}{h^{p_1}(w_1)} \norm{g}{h^{p_2}(w_2)} +  \norm{f}{h^{p_1}(w_1)}   \norm{J^s g}{h^{p_2}(w_2)} \quad \forall f, g \in \sw,
\end{equation}
which  supplements  some of the estimates obtained in \cite[Theorem 1.1]{MR3513582} for $J^s$ in a similar manner to that indicated in Remark~\ref{re:improvement}.  The case $w_1=w_2\equiv 1$ of \eqref{eq:kplochardy} was obtained in \cite{MR3750316} and is an extension and an improvement of \eqref{KPI:Lp}; indeed,  \eqref{eq:kplochardy} allows for $0<p,p_1,p_2<\infty$ and, when $1<p_1,p_2<\infty,$ it improves \eqref{KPI:Lp} by allowing the larger quantity $\norm{J^s(fg)}{h^p}$ on the left-hand side.

We note that the counterpart of Corollary~\ref{coro:KP:TL:B} gives in particular that $\itlw{p}{s}{q}{w}\cap L^\infty$ and $\ibesw{p}{s}{q}{w}\cap L^\infty$ are quasi-Banach algebras under pointwise multiplication for any $w\in A_\infty.$

\subsection{Applications to scattering of solutions to systems of PDEs}\label{sec:scattering} Our applications will be concerned with systems of differential equations on functions $u=u(t,x),$ $v=v(t,x)$ and $w=w(t,x),$ with $t\ge 0$ and $x\in\re^n,$ of the form
\begin{equation} \label{eq:a:b:}
\left\{ \begin{array}{lll}  \partial_t u =vw, & \partial_t v +a(D) v = 0, & \partial_t w + b(D) w = 0, \\
  u(0,x)=0,&v(0,x)=f(x),&w(0,x) = g(x),
 \end{array} \right.
\end{equation}
 where  $a(D)$ and $b(D)$ are (linear) Fourier multipliers with symbols  $a(\xi)$ and  $b(\xi),$ $\xi\in\re^n,$ respectively; that is, $\widehat{a(D)f}(\xi)=a(\xi)\widehat{f}(\xi)$ and $\widehat{b(D)f}(\xi)=b(\xi)\widehat{f}(\xi)$. As in B\'enyi et al.~\cite[Section 2.3]{MR3205530}, we formally have 
$$ 
v(t,x)= \int_{\re^n} e^{-t a(\xi)} \widehat{f}(\xi)\, \eixxi \dxi,\quad w(t,x)= \int_{\re^n}  e^{-t b(\eta)} \widehat{g}(\eta)\,\eixeta \deta,
$$
and
\begin{align*}
u(t,x) & = \int_0^t v(s,x) w(s,x) \,ds  = \int_{\rtn} \left(\int_0^t e^{-s (a(\xi)+b(\eta))} \,ds \right) \widehat{f}(\xi) \widehat{g}(\eta) \, \eixxe \dxi\deta.
\end{align*}
Setting $\lambda(\xi,\eta)=a(\xi)+b(\eta)$ and assuming that $\lambda$ never vanishes, the solution $u(t,x)$ can then be written as the action on $f$ and $g$ of the bilinear multiplier  with symbol $\frac{1-e^{-t\lambda(\xi,\eta)}}{\lambda(\xi,\eta)},$ that is,  
\begin{equation}\label{u:T:lambda}
u(t,x) = T_{\frac{1-e^{-t\lambda}}{\lambda}}(f,g)(x).
\end{equation}
Following Bernicot--Germain~\cite[Section 9.4]{MR2680189}, suppose there exists   $u_\infty\in \swp$ such that 
\begin{equation}\label{def:u:infty}
\lim\limits_{t \to \infty} u(t, \cdot ) = u_\infty \quad \text{in } \swp;
\end{equation}
then, given a function space $X$, we say that the solution $u$ of  \eqref{eq:a:b:} scatters in the function space $X$ if $u_\infty \in X.$

As an application of Theorems \ref{thm:CM:TL:B} and \ref{thm:ICM:TL:B} we obtain the following scattering properties for solutions to  systems of the type \eqref{eq:a:b:} involving powers of the Laplacian.

For $0<p_1,p_2, p, q\le \infty$ and $w_1,w_2\in A_\infty,$ set
\begin{align*}
\gamt&=2( [n(1/\min(p,q,1)+1/\min(1,p_1/\tau_{w_1},p_2/\tau_{w_2},q))]+1),\\
 \gamb&=2( [n(1/\min(p,q,1)+1/\min(1,p_1/\tau_{w_1},p_2/\tau_{w_2}))]+1).
\end{align*}

For $\delta>0$ define
$$
\Ss_{\delta}=\{(\xi,\eta)\in \re^{2n}: \abs{\eta}\le \delta^{-1}\abs{\xi}\text{ and }\abs{\xi}\le \delta^{-1}\abs{\eta}\}.
$$
\begin{theorem}\label{thm:scattering} Consider  $0 < p, p_1, p_2  \le \infty$  such that $\hcline$ and  $0 < q \leq \infty;$ let  $w_1,w_2\in A_\infty$ and set $w=w_1^{{p}/{p_1}} w_2^{{p}/{p_2}}.$ 
Fix $\gamma>0;$ if $\gamma$ is even, or $\gamma\ge \gamt$ in the setting of Triebel--Lizorkin spaces, or $\gamma\ge \gamb$ in the setting of Besov spaces, assume $f, g \in \swz;$ otherwise, assume that $f,g\in\swz$ are such that $\fhat(\xi)\ghat(\eta)$ is supported in $\Ss_{\delta}$ for some $0<\delta\ll1.$ Consider the system 
\begin{equation}\label{eq:Ds:Ds}
\left\{ \begin{array}{lll}  \partial_t u =vw, & \partial_t v +D^\gamma v = 0, & \partial_t w + D^\gamma w = 0, \\
  u(0,x)=0,&v(0,x)=f(x),&w(0,x) = g(x).
 \end{array} \right.
\end{equation}
If $0 < p,p_1,p_2 < \infty$ and  $s > \tau_{p,q}(w),$ the solution $u$ of \eqref{eq:Ds:Ds}  scatters in $\tlw{p}{s}{q}{w}$ to a function $u_\infty$ that satisfies the following estimates: 
\begin{equation}\label{eq:scattering1}
\norm{u_\infty}{\tlw{p}{s}{q}{w}} \lesssim \norm{f}{\tlw{p_1}{s-\gamma}{q}{w_1} } \norm{g}{H^{p_2}(w_2)} +  \norm{f}{H^{p_1}(w_1)}   \norm{g}{\tlw{p_2}{s-\gamma}{q}{w_2} },
\end{equation}
where the implicit constant is independent of $f$ and $g.$
If $0< p, p_1,p_2\leq \infty$ and $s > \tau_p(w)$, the solution $u$ of \eqref{eq:Ds:Ds}  scatters in $\besw{p}{s}{q}{w}$ to a function $u_\infty$ that satisfies the following estimates
\begin{equation}\label{eq:scattering2}
\norm{u_\infty}{\besw{p}{s}{q}{w}} \lesssim \norm{f}{\besw{p_1}{s-\gamma}{q}{w_1} } \norm{g}{H^{p_2}(w_2)} +  \norm{f}{H^{p_1}(w_1)}   \norm{g}{\besw{p_2}{s-\gamma}{q}{w_2} },
\end{equation}
where the Hardy spaces $H^{p_1}(w_1)$ and $H^{p_2}(w_2)$ must be replaced by $L^\infty$ if $p_1=\infty$ or $p_2=\infty,$ respectively, and the implicit constant is independent of $f$ and $g.$ 
If $w_1=w_2$ then different pairs of $p_1, p_2$ can be used on the right-hand sides of \eqref{eq:scattering1} and \eqref{eq:scattering2}; moreover, if $w\in A_\infty,$ then 
\begin{equation*}
\norm{u_\infty}{\tlw{p}{s}{q}{w}} \lesssim \norm{f}{\tlw{p}{s-\gamma}{q}{w} } \norm{g}{L^\infty} +  \norm{f}{L^\infty}   \norm{g}{\tlw{p}{s-\gamma}{q}{w}},
\end{equation*}
where $0<p<\infty,$ $0<q\le\infty,$  $s>\tau_{p,q}(w),$ and the implicit constant is independent of $f$ and $g.$ 
\end{theorem} 

For $\delta>0$ define
$$
\widetilde{\Ss}_{\delta}=\{(\xi,\eta)\in \re^{2n}: \abs{\eta}\le \delta^{-1}(1+\abs{\xi}^2)^{\frac{1}{2}}\text{ and }\abs{\xi}\le \delta^{-1}(1+\abs{\eta}^2)^{\frac{1}{2}}\}.
$$

\begin{theorem}\label{thm:scattering2} Consider  $0 < p, p_1, p_2  \le \infty$  such that $\hcline$ and  $0 < q \leq \infty;$ let  $w_1,w_2\in A_\infty$ and set $w=w_1^{{p}/{p_1}} w_2^{{p}/{p_2}}.$ Fix $\gamma>0;$ if $\gamma$ is even, or $\gamma\ge \gamt$ in the setting of Triebel--Lizorkin spaces, or $\gamma\ge \gamb$ in the setting of Besov spaces, assume $f, g \in \sw;$ otherwise, assume that $f,g\in\sw$ are such that $\fhat(\xi)\ghat(\eta)$ is supported in $\widetilde{\Ss}_{\delta}$ for some $0<\delta\ll1.$ Consider the system 
\begin{equation}\label{eq:Js:Js}
\left\{ \begin{array}{lll}  \partial_t u =vw, & \partial_t v +J^\gamma v = 0, & \partial_t w + J^\gamma w = 0, \\
  u(0,x)=0,&v(0,x)=f(x),&w(0,x) = g(x).
 \end{array} \right.
\end{equation}
If $0 < p,p_1,p_2 < \infty$ and  $s > \tau_{p,q}(w),$ the solution $u$ of \eqref{eq:Js:Js}  scatters in $\itlw{p}{s}{q}{w}$ to a function $u_\infty$ that satisfies the following estimates: 
\begin{equation}\label{eq:scattering21}
\norm{u_\infty}{\itlw{p}{s}{q}{w}} \lesssim \norm{f}{\itlw{p_1}{s-\gamma}{q}{w_1} } \norm{g}{h^{p_2}(w_2)} +  \norm{f}{h^{p_1}(w_1)}   \norm{g}{\itlw{p_2}{s-\gamma}{q}{w_2} },
\end{equation}
where the implicit constant is independent of $f$ and $g.$  
If $0< p, p_1,p_2\leq \infty$ and $s > \tau_p(w)$, the solution $u$ of \eqref{eq:Js:Js}  scatters in $\ibesw{p}{s}{q}{w}$ to a function $u_\infty$ that satisfies the following estimates
\begin{equation}\label{eq:scattering22}
\norm{u_\infty}{\ibesw{p}{s}{q}{w}} \lesssim \norm{f}{\ibesw{p_1}{s-\gamma}{q}{w_1} } \norm{g}{h^{p_2}(w_2)} +  \norm{f}{h^{p_1}(w_1)}   \norm{g}{\ibesw{p_2}{s-\gamma}{q}{w_2} },
\end{equation}
where the Hardy spaces $h^{p_1}(w_1)$ and $h^{p_2}(w_2)$ must be replaced by $L^\infty$ if $p_1=\infty$ or $p_2=\infty,$ respectively, and the implicit constant is independent of $f$ and $g.$  
If $w_1=w_2$ then different pairs of $p_1, p_2$ can be used on the right-hand sides of \eqref{eq:scattering21} and \eqref{eq:scattering22}; moreover, if $w\in A_\infty,$ then 
\begin{equation*}
\norm{u_\infty}{\itlw{p}{s}{q}{w}} \lesssim \norm{f}{\itlw{p}{s-\gamma}{q}{w} } \norm{g}{L^\infty} +  \norm{f}{L^\infty}   \norm{g}{\itlw{p}{s-\gamma}{q}{w}},
\end{equation*}
where $0<p<\infty,$ $0<q\le\infty,$  $s>\tau_{p,q}(w),$ and the implicit constant is independent of $f$ and $g.$  
\end{theorem}

\section{Preliminaries}\label{sec:prelim}

In this section we set some notation and present definitions and  results about weights, the scales of weighted Triebel--Lizorkin, Besov and Hardy spaces, and Coifman--Meyer multiplier operators.

The notations $\sw$ and $\swp$ are used for the Schwartz class of smooth rapidly decreasing functions defined on $\rn$ and its dual, the class of tempered distributions on $\rn$, respectively. $\swz$ refers to the closed subspace of functions in $\sw$ that have vanishing moments of all orders; that is, $f\in \swz$ if and only if $f\in \sw$ and $\int_{\rn}x^\alpha f(x)\,dx=0$ for all $\alpha\in \na_0^n.$ Its dual is  $\mathcal{S}'_0(\rn),$  which coincides with the class of tempered distributions modulo polynomials denoted by $\swp/\mathcal{P}(\rn).$ Throughout, all functions are defined on $\rn$ and therefore we omit $\rn$ in the notation of the  function spaces defined below.

 A weight on $\rn$ is a nonnegative, locally integrable function defined on $\rn$. Given $1<p<\infty,$ the Muckenhoupt class $A_p$ consists of all weights $w$ on $\rn$ such that  
 \begin{equation*}
 \sup_B\left(\frac{1}{\abs{B}}\int_Bw(x)\dx\right)\left(\frac{1}{\abs{B}}\int_Bw(x)^{-\frac{1}{p-1}}\dx\right)^{p-1}<\infty,
 \end{equation*}
where the supremum is taken over all Euclidean balls $B\subset \rn$ and $\abs{B}$ means the Lebesgue measure of $B;$ it follows that $A_p\subset A_{q}$ if $p\le q.$ We set $A_\infty=\bigcup_{p>1} A_p$ and  recall that, for $w\in A_\infty,$  $\tau_w=\inf\{\tau\in (1,\infty): w\in A_\tau\}.$ Note that  the conditions $0<r<p$ and $w\in A_{p/r}$  are equivalent to stating that  $0<r<p/\tau_w.$

 Given  $w\in A_\infty$  and $0<p\le\infty,$ we denote by $\lebw{p}{w}$ the space of measurable functions defined on $\rn$ such that 
 $$
\norm{f}{\lebw{p}{w}}=\left(\int_\rn\abs{f(x)}^pw(x)\dx\right)^{\frac{1}{p}}<\infty,
 $$
 with the corresponding change when $p=\infty.$ When $w=1,$ we simply write $L^p.$ Note that $L^\infty(w)=L^\infty$ for all $w\in A_\infty.$
 
 For a locally integrable function $f$ defined on $\rn,$ $\M (f)$  denotes the Hardy-Littlewood maximal function of $f$, that is
  \begin{equation*}
 \M (f)(x)=\sup_{x\in B}\frac{1}{\abs{B}}\int_B\abs{f(y)}dy\quad \forall x\in \rn,
 \end{equation*}
 where the supremum is taken over all Euclidean balls $B\subset\rn$ containing $x.$  Moreover, for $0<r<\infty,$ we set  $\M_r(f)=(\M(\abs{f}^r))^{1/r}.$

 We recall that if $1<p<\infty,$ then $\M$ is bounded on $\lebw{p}{w}$ if and only if $w\in A_p.$ In particular, $\M_r$ is bounded on $\lebw{p}{w}$ for $0<r<p$ and $w\in A_{p/r}$ (i.e. $0<r<p/\tau_w$). We will  also use the following vector-valued version of such result,  the weighted Fefferman-Stein inequality:
If $0<p<\infty,$ $0<q\le \infty,$  $0<r <\min(p,q)$ and $w \in A_{p/r}$ (i.e. $0<r<\min(p/\tau_w,q)$), then for all sequences $\{f_{j}\}_{j\in\ent}$ of locally integrable functions defined on $\rn,$ we have
 \begin{equation*}
 \norm{\left(\sum_{j\in\ent}\abs{\M_r (f_j)}^q\right)^{\frac{1}{q}}}{\lebw{p}{w}}\lesssim
 \norm{\left(\sum_{j\in\ent}\abs{f_j}^q\right)^{\frac{1}{q}}}{\lebw{p}{w}},\label{eq:wFS}
 \end{equation*}
where the implicit constant depends on $r,$ $p,$ $q,$ and $w$ and the summation in $j$ should be replaced by the supremum in $j$ if $q=\infty.$

The Fourier transform of a tempered distribution $f\in\swp$ is denoted by   $\fhat$; in particular, for $f\in L^1,$ we use the formula
\begin{equation*}
\widehat{f}(\xi)= \int_{\rn} f(x) e^{-2\pi i \xi\cdot x} \dx\quad \forall \xi\in\rn.
\end{equation*}
If $j\in \ent$ and  $h\in \sw,$ the operator $P_j^h$ is defined so that $\widehat{P_j^h f}(\xi)=h(2^{-j}\xi)\fhat(\xi)$ for $f\in \sw$ and $\xi\in \rn.$  If $\widehat{h}$ is supported in an annulus centered at the origin we will use the notation $\Delta_j^h$ rather than $P_j^h;$ if  $\widehat{h}$ is supported in a ball centered at the origin and $\widehat{h}(0)\neq 0,$   $S_j^h$ will be used instead of  $P_j^h.$
For $y\in\rn$ denote by $\tau_y$ the operator given by $\tau_y h(x)=h(x+y)$ for  $x\in\rn.$

We next record a lemma that  will be useful in the proof of the main results.

\begin{lemma}\label{lem:pointineq} Let $\phi_1,\phi_2\in \sw$ be  such that $\widehat{\phi_1}$ and $\widehat{\phi_2}$ have compact supports and  $\widehat{\phi_1}\widehat{\phi_2}=\widehat{\phi_1}.$  If $0<r\le 1$ and $\varepsilon>0,$ it holds that
\begin{align*}
\abs{P^{\tau_a\phi_1}_{j}f(x)}\lesssim (1+\abs{a})^{\varepsilon+\frac{n}{r}} \M_r(P^{\phi_2}_{j}f)(x)\quad \forall x,a\in\rn, j\in\ent, f\in\sw.
\end{align*}
\end{lemma}

\begin{proof} The estimate is a consequence of Lemma~\ref{coro:2:11} in  Appendix~\ref{sec:appendix} as we next show. In view of the supports of $\widehat{\phi_1}$ and $\widehat{\phi_2}$ we have $P^{\tau_a\phi_1}_{j}f=P^{\tau_a\phi_1}_{j}P^{\phi_2}_{j}f$ for  $j\in \ent $ and $f\in \sw.$ Applying Lemma~\ref{coro:2:11} with $\phi(x)=2^{nj}\tau_a \phi_1(2^j x),$ $A=2^j,$ $R\ge 1$ such that $\supp(\widehat{\phi_2})\subset \{\xi\in\rn:\abs{\xi}\le R\}$ and $d=\varepsilon+n/r,$ we get
\begin{align*}
\abs{P^{\tau_a\phi_1}_{j}f(x)}&\lesssim R^{n(\frac{1}{r}-1)} 2^{-jn}\norm{(1+\abs{2^j\cdot})^{\varepsilon+\frac{n}{r}}2^{nj}\tau_a \phi_1(2^j\cdot)}{L^\infty}\M_r(P^{\phi_2}_{j}f)(x)\\
&\sim \norm{(1+\abs{2^j\cdot})^{\varepsilon+\frac{n}{r}}\tau_a \phi_1(2^j\cdot)}{L^\infty}\M_r(P^{\phi_2}_{j}f)(x)\quad \forall x,a\in \rn, j\in \ent, f\in\sw.
\end{align*}
Since $\phi_1\in \sw,$ 
\[
\abs{\tau_a \phi_1(2^jx)}=\abs{\phi_1(2^j x +a)}\lesssim \frac{(1+\abs{a})^{\varepsilon+\frac{n}{r}}}{(1+\abs{2^jx})^{\varepsilon+\frac{n}{r}}}\quad \forall x,a\in \re,j\in \ent.
\]
Putting altogether we obtained the desired result.
\end{proof}

\subsection{Weighted Triebel--Lizorkin, Besov and Hardy spaces}\label{sec:spaces} 
Let $\psi$ and $\f$ denote functions in $\sw$ satisfying  the following conditions:
\begin{align*}
 &\supp(\widehat{\psi})\subset \{\xi\in\rn: \fr{1}{2}<\abs{\xi}<2\},\\
 &|\widehat{\psi}(\xi)|>c\quad {\text{ for all }}\xi \text{ such that } \fr{3}{5}<\abs{\xi}<\fr{5}{3} \text{ and some } c>0,\\
 &\supp(\widehat{\f})\subset \{\xi\in\rn: \abs{\xi}<2\},\\
 &|\widehat{\f}(\xi)|>c\quad {\text{ for all }}\xi \text{ such that } \abs{\xi}<\fr{5}{3} \text{ and some } c>0.\\
 \end{align*}
For $s\in\re,$  $0<p<\infty,$  $0<q\le \infty$ and $w\in A_\infty,$   the weighted homogeneous  Triebel-Lizorkin space $\tlw{p}{s}{q}{w}$ consists of all $f\in \swp/\mathcal{P}(\rn)$ such that 
\begin{equation*}
\norm{f}{\tlw{p}{s}{q}{w}}=\norm{\left(\sum_{j\in\ent}(2^{sj}|\Delta^\psi_jf|)^q\right)^{\frac{1}{q}}}{\lebw{p}{w}}<\infty.
\end{equation*}
 Similarly, the weighted inhomogeneous  Triebel-Lizorkin space   $\itlw{p}{s}{q}{w}$  is the class of all $f\in \swp$ such that 
\begin{equation*}
\norm{f}{\itlw{p}{s}{q}{w}}=\norm{S_0^{\f}f}{\lebw{p}{w}}+\norm{\left(\sum_{j\in\na}(2^{sj}|\Delta^\psi_jf|)^q\right)^{\frac{1}{q}}}{\lebw{p}{w}}<\infty.
\end{equation*}
 In both cases, the summation in $j$  is replaced by the supremum in $j$ if $q=\infty.$   For $s\in\re,$  $0<p,q\le \infty,$ and  $w\in A_\infty,$ the weighted homogeneous and inhomogeneous  Besov spaces are denoted by $\besw{p}{s}{q}{w}$ and $\ibesw{p}{s}{q}{w},$ respectively. They are defined analogously to the weighted Triebel-Lizorkin spaces by interchanging the order of the quasi-norms  in $\ell^q$ and $\lebw{p}{w}.$ 
 
 The definitions above are independent of the choice of the functions $\psi$ and $\f$ and all weighted Triebel--Lizorkin and Besov spaces  are quasi-Banach spaces (Banach spaces for $p,q\ge 1$).  The classes $\swz$ and $\sw$ are contained, respectively, in the weighted homogeneous and  inhomogeneus Triebel--Lizorkin and Besov spaces, and are  dense for finite values of $p$ and $q;$    $\sw$ is also contained in $\tlw{p}{s}{q}{w}$ and $\besw{p}{s}{q}{w}$  for  $p>1,$ $0<q \le \infty,$  $s>0,$ and $w\in A_p.$  We recall the so-called lifting properties: for any $p,$ $q$ and $s$ as in the definitions and for any $w\in A_\infty,$   it follows that
  \begin{align*}
 & \norm{f}{\tlw{p}{s}{q}{w}}\simeq\norm{D^s f}{\tlw{p}{0}{q}{w}} \quad \text{ and } \quad \norm{f}{\itlw{p}{s}{q}{w}}\simeq\norm{J^s f}{\itlw{p}{0}{q}{w}},
 \end{align*}
 with a corresponding statement for  Besov spaces. The reader is directed to  classical references such as  Frazier--Jawerth~\cite{MR808825, MR1070037},  Frazier--Jawerth--Weiss~\cite{MR1107300},  Peetre~\cite{MR0380394, MR0461123}, Qui~\cite{MR676560} and Triebel~\cite{MR3024598}  for the theory of Triebel-Lizorkin  and Besov spaces. 

Let $\phi\in \sw$ be such that $\int_{\rn}\phi(x)\dx\neq 0.$ Given $0<p<\infty,$ the Hardy space $H^p(w)$ is defined as the class of tempered distributions such that
\[
\norm{f}{H^p(w)}:= \norm{\sup_{0<t<\infty} \abs{t^{-n}\phi(t^{-1}\cdot)\ast f}}{L^p(w)};
\]
 the local Hardy spaced $h^p(w)$ consists of all tempered distributions such that
\[
\norm{f}{h^p(w)}:= \norm{\sup_{0<t<1} \abs{t^{-n}\phi(t^{-1}\cdot)\ast f}}{L^p(w)} .
\]
It turns out that $H^p(w)\simeq\tlw{p}{0}{2}{w}$ and $h^p(w)\simeq\itlw{p}{0}{2}{w}$ for $0<p<\infty$ and $w\in A_\infty.$ Moreover, $h^p(w)\simeq L^p(w)\simeq H^p(w)$ for $1<p<\infty$ and $w\in A_p.$ See   \cite[Theorem 1.4 and Remark 4.5]{MR676560}. The lifting property and the latter observations imply that $\tlw{p}{s}{2}{w}\simeq \dot{W}^{s,p}(w)$  and $\itlw{p}{s}{2}{w}\simeq W^{s,p}(w)$ for $0<p<\infty$ and $w\in A_\infty,$ where $\dot{W}^{s,p}(w)$ and ${W}^{s,p}(w)$ are the weighted Sobolev spaces defined by 
\begin{align*}
&\dot{W}^{s,p}(w)=\{f:\rn\to \com: D^sf\in H^p(w)\}, \quad \norm{f}{\dot{W}^{s,p}(w)}=\norm{D^sf}{H^p(w)},\\
&{W}^{s,p}(w)=\{f:\rn\to \com: J^sf\in h^p(w)\}, \quad \norm{f}{{W}^{s,p}(w)}=\norm{J^sf}{h^p(w)}.
\end{align*}
When $1<p<\infty$ and $w\in A_p,$   $H^p(w)$ and $h^p(w)$ in the definitions of  $\dot{W}^{s,p}(w)$ and ${W}^{s,p}(w),$  respectively, are just $L^p(w).$

\subsection{Nikol'skij representations for weighted Triebel-Lizorkin and Besov spaces}

The next theorem states the Nikol'skij representations of weighted  homogeneous and inhomogeneous Triebel-Lizorkin and Besov spaces with weights in $A_\infty$. It represents a weighted version of   \cite[Theorem 3.7]{MR837335} (see also \cite[Section 2.5.2]{MR3024598}), where the unweighted inhomogeneous case was studied. For  completeness, a sketch of its proof is outlined in Appendix~\ref{sec:appendix}.

\begin{theorem}\label{thm:Nikolskij:weighted} For $\A> 0,$ let $\{u_j\}_{j \in \ent} \subset \mathcal{S}'(\rn)$ be a sequence of tempered distributions such that
\begin{equation*}
\supp(\widehat{u_j}) \subset B(0, \A\, 2^j ) \quad \forall j \in \ent.
\end{equation*}
If $w\in A_\infty,$ then the following holds:  
\begin{enumerate}[(i)]
\item\label{item:thh:Nikolskij:TL} Let $0 < p < \infty$, $0 < q \leq \infty$ and $s > \tau_{p,q}(w)$. If $\norm{\{2^{js} u_j\}_{j\in\ent}}{L^p(w)(\ell^{q})} < \infty$, then the series $\sum_{j \in \ent} u_j$ converges in $\tlw{p}{s}{q}{w}$ (in $\mathcal{S}'_0(\rn)$ if $q=\infty$) and 
\begin{equation*}
\norm{\sum_{j \in \ent} u_j}{\tlw{p}{s}{q}{w}} \lesssim  \norm{\{2^{js} u_j\}_{j\in\ent}}{L^p(w)(\ell^{q})},
\end{equation*}
where the implicit constant depends only on $n,$ $\A,$ $s,$ $p$ and  $q.$  An analogous statement, with $j\in\naz,$ holds true for $\itlw{p}{s}{q}{w}$ (when $q=\infty,$  the convergence is in $\swp$).
\item\label{item:thh:Nikolskij:B} Let $0 < p, q \leq \infty$ and $s > \tau_p(w)$. If $\norm{\{2^{js} u_j\}_{j\in\ent}}{\ell^{q}(L^p(w))} < \infty$, then the series $\sum_{j \in \ent} u_j$ converges in  $\besw{p}{s}{q}{w}$ (in $\mathcal{S}'_0(\rn)$ if $q=\infty$) and 
\begin{equation*}
\norm{\sum_{j \in \ent} u_j}{\besw{p}{s}{q}{w}} \lesssim  \norm{\{2^{js} u_j\}_{j\in\ent}}{\ell^{q}(L^p(w))},
\end{equation*}
where the implicit constant depends only on $n,$ $\A,$ $s,$ $p$ and $q.$   An analogous statement, with $j\in\naz,$ holds true for $\ibesw{p}{s}{q}{w}$ (when $q=\infty,$  the convergence is in $\swp$).
\end{enumerate}
\end{theorem}

\subsection{Decomposition of Coifman--Meyer bilinear multiplier operators}\label{sec:decomp} For $m\in \re,$ let $\sigma$ be a Coifman--Meyer multiplier of order $m.$
Fix $\Psi \in \sw$ such that 
$$
\supp(\widehat{\Psi}) \subseteq \{ \xi \in \rn : \fr{1}{2} < |\xi| < 2 \} \quad\text{ and }\quad
\sum_{j\in\ent} \widehat{\Psi}(2^{-j}\xi) = 1 \,\,\forall \xi \in \rn \setminus \{ 0 \};
$$
define $\Phi \in \sw$ so that
$$
\widehat{\Phi}(0) \coloneqq 1,\quad \widehat{\Phi}(\xi) \coloneqq \sum_{j \le 0} \widehat{\Psi}(2^{-j}\xi)\quad \forall \xi\in\rn\setminus \{0\}.$$
For $a,b\in\rn,$ $\Do{\tau_a \Psi }{j} f$ and  $\So{ \tau_a \Phi }{j} f$ satisfy $\widehat{\Do{\tau_a \Psi }{j} f}(\xi)=\widehat{\tau_a\Psi}(2^{-j}\xi)\widehat{f}(\xi)=e^{2\pi i 2^{-j}\xi\cdot a} \widehat{\Psi}(2^{-j}\xi)\widehat{f}(\xi)$ and   $\widehat{\So{\tau_b \Phi }{j} f}(\xi)=\widehat{\tau_b\Phi}(2^{-j}\xi)\widehat{f}(\xi)=e^{2\pi i 2^{-j}\xi\cdot b} \widehat{\Phi}(2^{-j}\xi)\widehat{f}(\xi).$ 
By the work of Coifman and Meyer in \cite{MR518170},  given $N\in \na$ such that $N>n,$ it follows that $T_\sigma= T_\sigma^1 + T_\sigma^2$, where, for  $f\in \swz$ ($f\in \sw$ if $m\ge 0$) and $g\in\sw,$
\begin{align}\label{eq:decompT1}
T_\sigma^1(f,g)(x) &= \sum_{a,b\in \ent^n} \frac{1}{(1+|a|^2+|b|^2)^N} \sum_{j\in\ent} \C_j(a,b) \,(\Do{ \tau_a \Psi }{j} f)(x)\, (\So{ \tau_b \Phi }{j} g )(x)\quad \forall x\in \rn,
\end{align}
 the coefficients $\C_j(a,b)$   satisfy
\begin{equation}\label{eq:cjbound}
\abs{\C_j(a,b)}\lesssim 2^{jm}\quad \forall a,b\in\ent^n, j\in \ent,
\end{equation}
with the implicit constant depending on $\sigma,$ and an analogous expression holds for $T_\sigma^2$ with the roles of $f$ and $g$ interchanged. 

If $\sigma$ is an inhomogeneous Coifman--Meyer multiplier of order $m,$ a similar decomposition to \eqref{eq:decompT1} follows  with the summation in $j\in\naz$ rather than $j\in\ent,$  with  $\Delta_0^{\tau_a \Psi}$ replaced by $S_0^{\tau_a\Phi}$ and for $f,g\in \sw.$

\begin{remark}\label{re:numderiv1} For the formula \eqref{eq:decompT1}  and its corresponding counterpart for $T^2_\sigma$ to hold, the condition \eqref{eq:CMm} on the derivatives of $\sigma$ is only needed for multi-indices $\alpha$ and $\beta$ such that  $\abs{\alpha+\beta}\le 2N.$
\end{remark}

\begin{remark} We refer the reader to \cite[Lemma 2.1]{MR3750234} for decompositions in the spirit of \eqref{eq:decompT1} for the larger class of symbols $\dot{BS}^m_{1,1}$ :  such symbols may depend on the space variable $x,$ that is, $\sigma=\sigma(x,\xi,\eta)$ for $x,\xi,\eta\in\rn,$ and are such that for any multi-indices $\gamma,\alpha,\beta\in \na_0^n,$ there exists $C_{\gamma,\alpha,\beta}>0$ such that
\[
\abs{\partial_x^\gamma\partial_\xi^\alpha\partial_\eta^\beta\sigma(x,\xi,\eta)}\le C_{\gamma,\alpha,\beta} (\abs{\xi}+\abs{\eta})^{m+\abs{\gamma}-\abs{\alpha+\beta}}\quad \forall (\xi,\eta)\in \re^{2n}\setminus \{(0,0)\}.
\]  
Note that Coifman--Meyer multipliers of order $m$  belong to $\dot{BS}^m_{1,1}.$ 
\end{remark}

\section{Proofs of Theorems \ref{thm:CM:TL:B} and \ref{thm:ICM:TL:B}}\label{sec:mainproofs}

We only prove Theorem~\ref{thm:CM:TL:B}; the proof of Theorem~\ref{thm:ICM:TL:B} follows along the same lines.

\begin{proof}[Proof of Theorem~\ref{thm:CM:TL:B}] Consider $\Phi,$ $\Psi,$ $T_\sigma^1,$ $T_\sigma^2,$ $\{\C_j(a,b)\}_{j\in\ent,a,b\in\ent^n}$ as in Section~\ref{sec:decomp}. Let $m,$ $\sigma,$ $p,$ $p_1,$ $p_2,$ $q,$ $s,$ $w_1,$ $w_2$ and $w$ be as in the hypotheses.  
For ease of notation, $p_1$ and $p_2$ will be assumed to be finite; the same proof applies for \eqref{KP:CM:B} if that is not the case, and for \eqref{KP:CM:TL2}.

We next prove \eqref{KP:CM:TL} and \eqref{KP:CM:B}. 
 By symmetry,  it is enough to work with $T_\sigma^1$ and prove that 
 \begin{align*}
 \norm{T^{1}_\sigma(f,g)}{\dot{F}^s_{p,q}(w)} \lesssim  \norm{f}{\dot{F}^{s+m}_{p_1, q}(w_1)} \norm{g}{H^{p_2}(w_2)}\quad  \text{ and }\quad 
 \norm{T^1_\sigma(f,g)}{\dot{B}^s_{p,q}(w)} \lesssim  \norm{f}{\dot{B}^{s+m}_{p_1, q}(w_1)} \norm{g}{H^{p_2}(w_2)}.
\end{align*}
  Moreover, since $\norm{\sum f_j}{\tlw{p}{s}{q}{w}}^{\min(p,q,1)}\lesssim \sum\norm{f_j}{\tlw{p}{s}{q}{w}}^{\min(p,q,1)}$  and similarly for $\besw{p}{s}{q}{w}$, it suffices to prove that, given $\varepsilon>0$ there exist $0<r_1,r_2\le 1$  such that for all $g\in \sw$ and  $f\in \swz$ ($f\in \sw\cap \tlw{p}{s}{q}{w}$ or  $f\in \sw\cap \besw{p}{s}{q}{w}$ if $m\ge 0$),  it holds that
\begin{align}
 \norm{T^{a,b}(f,g)}{\dot{F}^s_{p,q}(w)} \lesssim (1+\abs{a})^{\varepsilon+\frac{n}{r_1}}  (1+\abs{b})^{\varepsilon+\frac{n}{r_2}} \norm{f}{\dot{F}^{s+m}_{p_1, q}(w_1)} \norm{g}{H^{p_2}(w_2)}\label{eq:estbbTL},\\
 \norm{T^{a,b}(f,g)}{\dot{B}^s_{p,q}(w)} \lesssim (1+\abs{a})^{\varepsilon+\frac{n}{r_1}}  (1+\abs{b})^{\varepsilon+\frac{n}{r_2}} \norm{f}{\dot{B}^{s+m}_{p_1, q}(w_1)} \norm{g}{H^{p_2}(w_2)}\label{eq:estbbB},
\end{align}
where
\[
T^{a,b}(f,g):=\sum_{j\in\ent} \C_j(a,b) \,(\Do{ \tau_a \Psi }{j} f)\, (\So{ \tau_b \Phi }{j} g )
\]
and the implicit constants are independent of $a$ and $b.$  We will assume $q$ finite; obvious changes apply if that is not the case.

In view of the supports of $\Psi$ and $\Phi$ we have that 
\begin{equation*}
\supp (\mathcal{F}[\C_j(a,b) \,(\Do{ \tau_a \Psi }{j} f ) \, ( \So{ \tau_b \Phi }{j} g )])  \subset \{\xi \in \rn: |\xi| \lesssim 2^j \} \quad \forall j \in \ent,\,a,b\in \ent^n.
\end{equation*}

For \eqref{eq:estbbTL},
Theorem \ref{thm:Nikolskij:weighted}\eqref{item:thh:Nikolskij:TL}, the bound \eqref{eq:cjbound} for $\C_j(a,b)$, and H\"older's inequality  imply
\begin{align*}
\norm{T^{a,b}(f,g)}{\dot{F}^s_{p,q}(w)} & \lesssim \norm{\{2^{sj} \C_j(a,b) \,(\Do{ \tau_a \Psi }{j} f ) \, (\So{ \tau_b \Phi }{j} g)\}_{j\in\ent} }{L^p(w)(\ell^q)}\\
& \lesssim \norm{\left(\sum\limits_{j \in \ent}  2^{(s +m) q j}  |(\Do{ \tau_a \Psi }{j} f )(x) \, (\So{ \tau_b \Phi }{j} g)|^q   \right)^\frac{1}{q}}{L^p(w)}\\
& \le\norm{\sup\limits_{j \in \ent} |(\So{ \tau_b \Phi }{j} g)| \left(\sum\limits_{j \in \ent}  2^{(s +m) q j}  |(\Do{ \tau_a \Psi }{j} f )|^q   \right)^\frac{1}{q}}{L^p(w)}\\
& \le \norm{\left(\sum_{j \in \ent}  2^{(s +m) q j}  |\Do{ \tau_a \Psi }{j} f |^q   \right)^\frac{1}{q}}{L^{p_1}(w_1)} \norm{\sup\limits_{j \in \ent} |\So{ \tau_b \Phi }{j} g|}{L^{p_2}(w_2)}.
\end{align*}
Consider $\varphi,\psi\in\sw$ as in Section~\ref{sec:spaces} such that   $\widehat{\varphi}\equiv 1$ on $\supp(\widehat{\Phi})$ and  $\widehat{\psi}\equiv 1$ on $\supp(\widehat{\Psi}).$  Let   $0<r_1<\min(1, p_1/\tau_{w_1},q)$; by Lemma~\ref{lem:pointineq} and the weighted Fefferman-Stein inequality  we have that  
\begin{align*}
\norm{\left(\sum_{j \in \ent}  2^{(s +m) q j}  |(\Do{ \tau_a \Psi }{j} f )|^q   \right)^\frac{1}{q}}{L^{p_1}(w_1)}&\lesssim (1+\abs{a})^{\varepsilon+\frac{n}{r_1}}
\norm{\left(\sum_{j \in \ent}  2^{(s +m) q j}  |\M_{r_1}(\Do{\psi }{j} f) |^q   \right)^\frac{1}{q}}{L^{p_1}(w_1)}\\
&\lesssim (1+\abs{a})^{\varepsilon+\frac{n}{r_1}} \norm{\left(\sum_{j \in \ent}  2^{(s +m) q j}  |\Do{\psi }{j} f|^q   \right)^\frac{1}{q}}{L^{p_1}(w_1)}\\
&\sim (1+\abs{a})^{\varepsilon+\frac{n}{r_1}}  \norm{f}{\dot{F}^{s+m}_{p,q}(w_1)},
\end{align*}
where the implicit constants are independent of $a$ and $f.$ Next, let  $0<r_2<\min(1,p_2/\tau_{w_2})$; by Lemma~\ref{lem:pointineq} and the boundedness properties of the Hardy-Littlewood maximal operator on weighted Lebesgue space  we have that  
\begin{align*}
\norm{\sup_{j\in\ent}|\So{\tau_b\Phi}{j}g|}{L^{p_2}(w_2)}&\lesssim (1+\abs{b})^{\varepsilon+\frac{n}{r_2}} \norm{\M_{r_2}(\sup_{j\in\ent}|\So{\varphi}{j}g|)}{L^{p_2}(w_2)}\\
&\lesssim (1+\abs{b})^{\varepsilon+\frac{n}{r_2}} \norm{\sup_{j\in\ent}|\So{\varphi}{j}g|}{L^{p_2}(w_2)}\\
&\sim (1+\abs{b})^{\varepsilon+\frac{n}{r_2}} \norm{g}{H^{p_2}(w_2)},
\end{align*}
where the implicit constants are independent of $b$ and $g.$ Putting all together we obtain \eqref{eq:estbbTL}.

For \eqref{eq:estbbB},  Theorem \ref{thm:Nikolskij:weighted}\eqref{item:thh:Nikolskij:B}, the bound \eqref{eq:cjbound} for $\C_j(a,b)$ and H\"older's inequality  give
\begin{align*}
\norm{T^{a,b}(f,g)}{\dot{B}^s_{p,q}(w)} & \lesssim \norm{\{2^{sj} \C_j(a,b) \,(\Do{ \tau_a \Psi }{j} f ) \, (\So{ \tau_b \Phi }{j} g)\}_{j\in\ent} }{\ell^q(L^p(w))}\\
& \lesssim \left(\sum\limits_{j \in \ent}  2^{(s +m) q j}  \norm{(\Do{ \tau_a \Psi }{j} f ) \, (\So{ \tau_b \Phi }{j} g)}{L^p(w)}^q   \right)^\frac{1}{q}  \\
&\le  \left(\sum\limits_{j \in \ent}  2^{(s +m) q j}  \norm{(\Do{ \tau_a \Psi }{j} f ) }{L^{p_1}(w_1)}^q   \right)^\frac{1}{q}  \norm{\sup_{k\in \ent}|\So{ \tau_b \Phi }{k} g|}{L^{p_2}(w_2)}\\
& \lesssim  (1+\abs{a})^{\varepsilon+\frac{n}{r_1}}  (1+\abs{b})^{\varepsilon+\frac{n}{r_2}}  \norm{f}{\dot{B}^{s+m}_{p_1, q}(w_1)} \norm{g}{H^{p_2}(w_2)},
\end{align*}
where in the last inequality we have used Lemma~\ref{lem:pointineq} and the boundedness properties of $\M$ with  $0<r_j<\min(1,p_j/\tau_{w_j})$ for $j=1,2$ .

It is clear from the proof above that if $w_1=w_2,$ then  different pairs of $p_1, p_2$ related to $p$ through the H\"older condition can be used on the right-hand sides of \eqref{KP:CM:TL} and \eqref{KP:CM:B}; in such case $w=w_1=w_2.$  
\end{proof}

\begin{remark} \label{re:numderiv2} 
For convergence purposes, the relations between $N$ in \eqref{eq:decompT1} and the powers  $\varepsilon+n/r_1$ and $\varepsilon+n/r_2$ in \eqref{eq:estbbTL} and   \eqref{eq:estbbB} must be such that $(N-\varepsilon-n/r_1)\,r^*>n$ and $(N-\varepsilon-n/r_2)\,r^*>n,$ where $r^*=\min(p,q,1).$ Moreover,  $r_1$ and $r_2$ were selected so that $0<r_j<\min(1, p_j/\tau_{w_j},q)$ in the context of  Triebel--Lizorkin spaces and  $0<r_j<\min(1,p_j/\tau_{w_j})$ in the context of Besov spaces. Therefore, if  $N>n(1/r^*+1/\min(1, p_1/\tau_{w_1},p_2/\tau_{w_2},q))$ in the Triebel--Lizorkin setting and $N>n(1/r^*+1/\min(1, p_1/\tau_{w_1},p_2/\tau_{w_2}))$ in the Besov setting,  $\varepsilon,$  $r_1$ and  $r_2$ can be chosen so that all the conditions above are satisfied. In view of this and Remark~\ref{re:numderiv1},  the multiplier $\sigma$ in Theorem~\ref{thm:CM:TL:B}  needs only satisfy  \eqref{eq:CMm} for    $\abs{\alpha+\beta}\le 2( [n(1/r^*+1/\min(1,p_1/\tau_{w_1},p_2/\tau_{w_2},q))]+1)=\gamt$ in the Triebel--Lizorkin case and $\abs{\alpha+\beta}\le 2( [n(1/r^*+1/\min(1, p_1/\tau_{w_1},p_2/\tau_{w_2}))]+1)=\gamb$ in the Besov case. An analogous observation follows for the multiplier $\sigma$ in Theorem~\ref{thm:ICM:TL:B} in relation to the condition  obtained from \eqref{eq:CMm} with $\abs{\xi}+\abs{\eta}$ replaced by $1+\abs{\xi}+\abs{\eta}.$ 
  \end{remark}

\section{Proofs of Theorems \ref{thm:scattering} and \ref{thm:scattering2}} \label{sec:applications}

\begin{proof}[Proof of Theorem~\ref{thm:scattering}] Using the notation from Section~\ref{sec:scattering},  we have $a(\xi)=\abs{\xi}^\gamma$ and $b(\eta)=\abs{\eta}^\gamma;$ therefore, $\lambda(\xi,\eta)=\abs{\xi}^\gamma+\abs{\eta}^\gamma.$ Note that all corresponding integrals for $v(t,x),$ $w(t,x)$ and $u(t,x)$ are absolutely convergent for $t>0,$ $x\in\rn$ and $f,g\in \sw.$ If we further assume that $f,g\in \swz,$ the Dominated Convergence Theorem implies that $u(t,\cdot)\to u_\infty$  both pointwise and in $\swp,$ where
$$
u_\infty(x)=\int_{\re^{2n}} (a(\xi)+b(\eta))^{-1}\fhat(\xi)\ghat(\eta)\eixxe\dxi\deta=T_{\lambda^{-1}} (f,g)(x).
$$

If $\gamma$ is an even positive integer then $\lambda^{-1}$ satisfies the estimates \eqref{eq:CMm} with $m=-\gamma$ for all $\alpha,\beta\in\na_0^n.$ Then,  all estimates from Theorem~\ref{thm:CM:TL:B} hold for $T_{\lambda^{-1}}$  and therefore the desired estimates follow for $u_\infty$ with constants independent of $f,g\in\swz.$ 

Let $p_1,p_2,p,q, w_1,w_2$ be as in the hypotheses.  If $\gamma>0$ and $\gamma$ is not an even integer, then $\lambda^{-1}$ satisfies the estimates \eqref{eq:CMm} with $m=-\gamma$ as long as  $\alpha,\beta\in\na_0^n$ are such that $\abs{\alpha}<\gamma$ and $\abs{\beta}<\gamma;$ in particular, $\lambda^{-1}$ satisfies \eqref{eq:CMm} with $m=-\gamma$ for $\alpha,\beta\in \na_0^n$ such that $\abs{\alpha+\beta}<\gamma.$
In view of Remark~\ref{re:numderiv2}, all estimates from Theorem~\ref{thm:CM:TL:B} hold for $T_{\lambda^{-1}}$ if $\gamma\ge \gamt$ in the context of Triebel--Lizorkin spaces and if $\gamma\ge \gamb$ in the context of Besov spaces; as a consequence, the desired estimates follow for $u_\infty$ with constants independent of $f,g\in\swz$ for such values of $\gamma.$

On the other hand, if $0<\gamma<\gamt$ in the Triebel-Lizorkin space setting or $0<\gamma<\gamb$ in the Besov space setting, and $\gamma$ is not an even positive integer, consider  $h\in\Ss(\re^{2n}) $ such that $\supp(h)\subset \Ss_{\delta/2}$ and $h\equiv 1$ on $\Ss_{\delta}.$ Then, for $f,g\in \swz$ such that $\fhat(\xi)\ghat(\eta)$ is supported in $\Ss_{\delta}$ we have $h(\xi,\eta)\fhat(\xi)\ghat(\eta)=\fhat(\xi)\ghat(\eta);$   therefore, $T_{\lambda^{-1}}(f,g)=T_{\Lambda}(f,g),$ where
$\Lambda(\xi,\eta)=h(\xi,\eta)/(\abs{\xi}^\gamma+\abs{\eta}^\gamma).$ The multiplier $\Lambda$ verifies  \eqref{eq:CMm} with $m=-\gamma$  for all $\alpha,\beta\in\na_0^n$ (with constants that depend on $\delta$). Then all estimates from Theorem~\ref{thm:CM:TL:B} hold for $T_\Lambda$ and therefore the desired estimates follow for $u_\infty$ with constants dependent on $\delta$ and independent of $f,g\in\swz$ such that  $\fhat(\xi)\ghat(\eta)$ is supported in $\Ss_{\delta}.$
\end{proof}

\begin{proof}[Proof of Theorem~\ref{thm:scattering2}] We proceed as in the proof of Theorem~\ref{thm:scattering} with $\lambda(\xi,\eta)=(1+\abs{\xi}^2)^{\gamma /2}+(1+\abs{\eta}^2)^{\gamma /2}$ and an application of  Theorem~\ref{thm:ICM:TL:B}.  
\end{proof}

 \section{Leibniz-type rules  in other function space settings}\label{sec:more}
 
 In this section we illustrate the fact that the strategy applied in the proofs of Theorems~\ref{thm:CM:TL:B} and \ref{thm:ICM:TL:B} constitutes a unifying approach for obtaining Leibniz-type rules for Coifman--Meyer multipliers in a variety of function spaces. 
 
We start by isolating the main features associated to the weighted Triebel--Lizorkin and Besov spaces used for the proofs of Theorems~\ref{thm:CM:TL:B} and \ref{thm:ICM:TL:B}:
\begin{enumerate}[(i)]
\item\label{item:first} there exists $r>0$ such that $\norm{f+g}{\itlw{p}{s}{q}{w}}^r\le \norm{f}{\itlw{p}{s}{q}{w}}^r+\norm{g}{\itlw{p}{s}{q}{w}}^r;$ similarly for the weighted inhomogeneous  Besov spaces and the weighted homogeneous  Triebel--Lizorkin and Besov spaces;
\item \label{item:second} H\"older's inequality in weighted Lebesgue spaces;
\item \label{item:third}  the boundedness properties in weighted Lebesgue spaces  of the Hardy--Littlewood maximal operator (for the Besov space setting) and the weighted Fefferman--Stein inequality (for the Triebel--Lizorkin space setting);
\item \label{item:last} Nikol'skij representations for weighted Triebel--Lizorkin and Besov spaces (Theorem~\ref{thm:Nikolskij:weighted}).
\end{enumerate}

The method used to prove Theorems~\ref{thm:CM:TL:B} and \ref{thm:ICM:TL:B} can then be effectively  applied to other settings of function spaces of the form $\itl{\mathcal{X}_p}{s}{q}$ and $\ibes{\mathcal{X}_p}{s}{q}$ (or $\tl{\mathcal{X}_p}{s}{q}$ and $\bes{\mathcal{X}_p}{s}{q}$),  where $\mathcal{X}_p$ represents a quasi-Banach space within a given family indexed by $p\in \mathcal{I}$ ($\mathcal{I}$ is some suitable set), and $\itl{\mathcal{X}_p}{s}{q}$ and $\ibes{\mathcal{X}_p}{s}{q}$ are defined in the same way as $\itlw{p}{s}{q}{w}$ and $\ibesw{p}{s}{q}{w},$ respectively, with the  quasi-norm in $\lebw{p}{w}$ replaced by the quasi-norm in the space $\mathcal{X}_p.$ The family $\{\mathcal{X}_p\}_{p\in \mathcal{I}}$ is such that appropriate counterparts of the properties \eqref{item:first}-\eqref{item:last} hold true. As a consequence, versions of Theorems~\ref{thm:CM:TL:B} and \ref{thm:ICM:TL:B} as well as of Theorems~\ref{thm:scattering} and \ref{thm:scattering2} can be obtained in such contexts.

We illustrate the above in the cases where $\mathcal{X}_p$ corresponds to the scale of weighted Lorentz spaces, weighted Morrey spaces and variable-exponent Lebesgue spaces. In such contexts, we state counterparts of Theorem~\ref{thm:ICM:TL:B} for the corresponding inhomogeneous Triebel--Lizorkin spaces as  model results. The statements  for the inhomogeneous Besov-type spaces and for the counterparts of   Theorem~\ref{thm:scattering2}, as well as the details and statements for the homogeneous settings,  are left to the reader.

\subsection{Leibniz-type rules in the settings of Lorentz-based Triebel--Lizorkin and Besov spaces.}

Given $0<p<\infty$ and $0<t\le \infty$ or $p=t=\infty,$ and an $A_\infty$ weight $w$ defined on $\rn,$  we denote by $L^{p,t}(w)$ the weighted Lorentz space consisting of complex-valued, measurable functions $f$ defined on $\rn$ such that
\[
\|f\|_{\lebw{p,t}{w}}=\left(\int_0^\infty \left(\tau^{\frac{1}{p}} f_w^*(\tau)\right)^t\,\frac{d\tau}{\tau}\right)^{\frac{1}{t}}<\infty,
\]
where $f^*_w(\tau)=\inf\{\lambda\ge 0:w_f(\lambda)\le \tau\}$ with
$w_f(\lambda)=w(\{x\in\rn : \abs{f(x)}>\lambda\})$; the obvious changes apply if $t=\infty.$
 It follows that $\lebw{p,p}{w}=\lebw{p}{w}$ for $0<p\le \infty.$  We refer the reader to Hunt~\cite{MR0223874} for more details about Lorentz spaces. 
 
 The corresponding weighted inhomogeneous Triebel--Lizorkin and Besov spaces are denoted by $\itlw{(p,t)}{s}{q}{w}$ and  $\ibesw{(p,t)}{s}{q}{w},$  respectively. These spaces contain $\sw,$ are independent of the choice of $\varphi$ and $\psi$ from Section~\ref{sec:spaces}, are quasi-Banach spaces and have appeared in a variety of settings (see Seeger--Trebels~\cite{ST2018} and references therein). The space $h^{p,t}(w)$ is defined in the same way as $h^p(w)$ with the quasi-norm in $\lebw{p}{w}$ replaced by the quasi-norm in $\lebw{p,t}{w}.$

 We next consider the corresponding properties \eqref{item:first}-\eqref{item:last} in this context. Regarding property \eqref{item:first}, given $0<p<\infty,$  $0<t,q\le \infty$   and $s\in \re,$   it follows that there exist $r>0$ and a quasi-norm $|||\cdot |||_{\lebw{p,t}{w}(\ell^q)}$ comparable to $\|\cdot\|_{\lebw{p,t}{w}(\ell^q)}$ such that $|||\cdot |||^r_{\lebw{p,t}{w}(\ell^q)}$ is subadditive; this is an adecuate substitute for property \eqref{item:first}. The quasi-norm $|||\cdot |||_{\lebw{p,t}{w}(\ell^q)}$ is defined analogously to $\|\cdot \|_{\lebw{p,t}{w}(\ell^q)}$ by replacing $\|\cdot \|_{\lebw{p,t}{w}}$ with a comparable quasi-norm $|||\cdot |||_{\lebw{p,t}{w}}$ for which $|||\cdot |||^r_{\lebw{p,t}{w}}$ is subadditive (see \cite[p. 258,
  (2.2)]{MR0223874}). As for property \eqref{item:second}, weighted Lorentz spaces satisfy a  H\"older-type inequality (see \cite[Thm 4.5]{MR0223874}): Given a weight $w$ in $\rn$ and  indices $0<p, p_1,p_2<\infty$ and $ 0<t, t_1,t_2\le \infty$  such that $\hcline$ and $1/t=1/t_1+1/t_2,$ it holds that
  \begin{equation*}
  \norm{fg}{\lebw{p,t}{w}}\lesssim \norm{f}{\lebw{p_1,t_1}{w}} \norm{g}{\lebw{p_2,t_2}{w}},
  \end{equation*}
 where the implicit constant is independent of $f$ and $g$ ($p_1=t_1=\infty,$ which gives $p=p_2$ and $t_2=t,$ is also allowed). The following boundedness properties of the Hardy--Littlewood maximal operator in  weighted Lorentz spaces (property \eqref{item:third}) hold true: 
  If $0<p<\infty,$  $0<t,q\le\infty,$  $0<r<\min(p/\tau_w,q)$ and $0<r\le t,$ it holds that
\begin{equation}\label{eq:FTLorentz}
\norm{\left(\sum_{j\in\ent} \abs{\M_r(f_j)}^q\right)^{\frac{1}{q}}}{\lebw{p,t}{w}}\lesssim \norm{\left(\sum_{j\in\na_0} \abs{f_j}^q\right)^{\frac{1}{q}}}{\lebw{p,t}{w}}\quad \forall \{f_j\}_{j\in \na_0}\in \lebw{p,t}{w}(\ell^q);  
\end{equation}
in particular, if $0<r<p/\tau_w$ and $0<r\le t,$ it holds that 
\begin{equation*}
\norm{\M_r(f)}{\lebw{p,t}{w}}\lesssim \norm{f}{\lebw{p,t}{w}}\quad \forall f\in \lebw{p,t}{w}. 
\end{equation*}
When $r=1,$ $1<p<\infty,$  $1\le t\le \infty$ and $1<q\le \infty,$  the vector-valued inequality above follows from extrapolation and the weighted Fefferman--Stein inequality in weighted Lebesgue spaces (see   \cite[Theorem 4.10 and comments on p. 70]{MR2797562} for the extrapolation theorem).
The rest of the cases follow from the latter and the fact that $\norm{\abs{f}^s}{\lebw{p,t}{w}}=\norm{f}{\lebw{sp,st}{w}}^s$ for any $0<s<\infty.$ 
Regarding property \eqref{item:last}, the  Nikol'skij representation for $\itlw{(p,t)}{s}{q}{w}$  and $\ibesw{(p,t)}{s}{q}{w}$ with $w\in A_\infty$ can be stated as in  Theorem~\ref{thm:Nikolskij:weighted} with $0 < p< \infty,$ $0<t,q\le \infty$;  $s > (1/\min(p/\tau_w, t, q,1)-1)$ and $\itlw{p}{s}{q}{w}$ replaced with $\itlw{(p,t)}{s}{q}{w}$ in the Triebel--Lizorkin setting;  $s> \tau_{p,t}(w)$ and $\ibesw{p}{s}{q}{w}$ replaced with $\ibesw{(p,t)}{s}{q}{w}$ in the Besov setting. In the context of $\itlw{(p,t)}{s}{q}{w},$ the convergence of the series  holds in $\swp$ if $t=\infty$ or $q=\infty$ and in $\itlw{(p,t)}{s}{q}{w}$ otherwise; in the setting of $\ibesw{(p,t)}{s}{q}{w},$ the convergence of the series holds in $\swp$ if $q=\infty$ and in $\ibesw{(p,t)}{s}{q}{w}$ otherwise. The proofs follow parallel steps to those in the proof of Theorem~\ref{thm:Nikolskij:weighted} (see also Remark~\ref{re:app}).

As an exemplary result, we next present an analogue  to Theorem~\ref{thm:ICM:TL:B} in the context of the spaces $\itlw{(p,t)}{s}{q}{w}.$ For $w\in A_\infty,$ set $\tau_{p,t,q}(w):=n(1/\min(p/\tau_w, t, q, 1)-1).$

\begin{theorem}\label{thm:ICM:Lorentz}  For $m \in \re,$ let $\sigma(\xi,\eta),$ $\xi,\eta\in\rn,$ be an inhomogeneous Coifman-Meyer multiplier of order $m.$ If $w\in A_\infty,$ $0 < p, p_1, p_2 < \infty$ and $0 < t, t_1, t_2\le \infty$  are  such that $\hcline$ and $1/t=1/t_1+1/t_2,$  $0 < q \le \infty$  and  $s > \tau_{p,t,q}(w),$  it holds that
\begin{equation*}
\norm{T_\sigma(f,g)}{\itlw{(p,t)}{s}{q}{w}} \lesssim \norm{f}{\itlw{(p_1,t_1)}{s+m}{q}{w}} \norm{g}{h^{p_2,t_2}(w)} +  \norm{f}{h^{p_1,t_1}(w)}   \norm{g}{\itlw{(p_2,t_2)}{s+m}{q}{w}} \quad \forall f, g \in \sw.
\end{equation*}
Different pairs of $p_1,p_2$ and $t_1, t_2$ can be used on the right-hand side of the inequality above. 
Moreover, if $w\in A_\infty,$ $0<p<\infty,$ $0<t,q\le \infty$   and $s > \tau_{p,t,q}(w),$  it holds that
\begin{equation*}
\norm{T_\sigma(f,g)}{\itlw{(p,t)}{s}{q}{w}} \lesssim \norm{f}{\itlw{(p,t)}{s+m}{q}{w}} \norm{g}{L^\infty} +  \norm{f}{L^\infty}   \norm{g}{\itlw{(p,t)}{s+m}{q}{w}} \quad \forall f, g \in \sw.
\end{equation*}
\end{theorem}

The lifting property $ \norm{f}{\itl{(p,t)}{s}{q}(w)}\simeq\norm{J^s f}{\itl{(p,t)}{0}{q}(w)}$ holds true for $s \in \re,$ $0<p<\infty$ and  $0<t, q\le \infty;$ this is implied by the Fefferman--Stein inequality \eqref{eq:FTLorentz} through a proof analogous to that of the  lifting property of the standard Triebel--Lizorkin spaces $\itl{p}{s}{q}.$
Then, under the assumptions of Theorem~\ref{thm:ICM:Lorentz}  we obtain, in particular,
\begin{equation*}
\norm{J^s (fg)}{\itl{(p,t)}{0}{q}(w)} \lesssim \norm{J^s f}{\itl{(p_1,t_1)}{0}{q}(w)} \norm{g}{h^{p_2,t_2}(w)} +  \norm{f}{h^{p_1,t_1}(w)}   \norm{J^sg}{\itl{(p_2,t_2)}{0}{q}(w)};
\end{equation*}
\begin{equation*}
\norm{J^s(fg)}{\itlw{(p,t)}{0}{q}{w}} \lesssim \norm{J^sf}{\itlw{(p,t)}{0}{q}{w}} \norm{g}{L^\infty} +  \norm{f}{L^\infty}   \norm{J^s g}{\itlw{(p,t)}{0}{q}{w}}.
\end{equation*}
These last two estimates supplement the results in \cite[Theorem 6.1]{MR3513582}, where related Leibniz-type rules in Lorentz spaces were obtained.

 \subsection{Leibniz-type rules in the settings of Morrey-based Triebel--Lizorkin and Besov spaces.}

Given $0<p\le t<\infty$ and $w\in A_\infty,$  we denote by $\mow{p}{t}{w}$ the weighted Morrey space consisting of functions $f\in L^p_{\text{loc}}(\rn)$  such that
\[
\|f\|_{\mow{p}{t}{w}}=\sup_{B\subset \rn}w(B)^{\frac{1}{t}-\frac{1}{p}}\left( \int_B\abs{f(x)}^pw(x)\dx\right)^\frac{1}{p}<\infty,
\]
where the supremum is taken over all Euclidean balls $B$ contained in $\rn;$
 it easily  follows that $\mow{p}{p}{w}=\lebw{p}{w}.$  We refer the reader to the work Rosenthal--Schmeisser~\cite{MR3538648} and the  references it contains for more details about weighted Morrey spaces. The corresponding weighted inhomogeneous Triebel--Lizorkin spaces and inhomogeneous Besov spaces  are denoted by $\itlw{[p,t]}{s}{q}{w}$ and   $\ibesw{[p,t]}{s}{q}{w},$ respectively. These Morrey-based Triebel--Lizorkin and Besov  spaces are independent of the choice of $\varphi$ and $\psi$ given in  Section~\ref{sec:spaces} and are quasi-Banach spaces that contain $\sw$ (see the works Kozono--Yamazaki~\cite{MR1274547}, Mazzucato~\cite{MR1946395}, Izuki et al.~\cite{MR2792058} and the references they cotain). The corresponding local Hardy spaces are denoted by $h^t_p(w).$

Property \eqref{item:first} for $\itlw{[p,t]}{s}{q}{w}$ and   $\ibesw{[p,t]}{s}{q}{w}$ is easily verified with  $r=\min(p,q,1)$ using that $\norm{\abs{f}^s}{\mow{p}{t}{w}}=\norm{f}{\mow{sp}{st}{w}}^s$ for $0<s<\infty.$ Regarding property \eqref{item:second}, we have that if $0<p\le t<\infty,$ $0<p_1\le t_1<\infty$ and $0<p_2\le t_2<\infty$ are such that $\hcline$ and $1/t=1/t_1+1/t_2,$ then  
\begin{equation*}
\norm{fg}{\mow{p}{t}{w}}\le \norm{f}{\mow{p_1}{t_1}{w}} \norm{g}{\mow{p_2}{t_2}{w}};
\end{equation*}
also, if $0<p\le t<\infty,$ $0<p_1,p_2<\infty$ are such that $\hcline$ and $w=w_1^{p/p_1}w_2^{p/p_2}$ for weights $w_1$ and $w_2,$ then 
\begin{equation*}
\norm{fg}{\mow{p}{t}{w}}\le \norm{f}{\mow{p_1}{\frac{p_1t}{p}}{w_1}} \norm{g}{\mow{p_2}{\frac{p_2t}{p}}{w_2}}.
\end{equation*}
Both inequalities are  straightforward consequences of H\"older's inequality for weighted Lebesgue spaces. As for property \eqref{item:third}, it holds that  if $0<p\le t<\infty,$  $0<q\le \infty$ and  $0<r<\min(p/\tau_w,q),$ then
\begin{equation}\label{eq:FTMorrey}
\norm{\left(\sum_{j\in\ent} \abs{\M_r(f_j)}^q\right)^{\frac{1}{q}}}{\mow{p}{t}{w}}\lesssim \norm{\left(\sum_{j\in\na_0} \abs{f_j}^q\right)^{\frac{1}{q}}}{\mow{p}{t}{w}}\quad \forall \{f_j\}_{j\in \na_0}\in \mow{p}{t}{w}(\ell^q);  
\end{equation}
in particular, if $0<p\le t<\infty$ and $0<r<p/\tau_w,$ it holds that 
\begin{equation*}
\norm{\M_r(f)}{\mow{p}{t}{w}}\lesssim \norm{f}{\mow{p}{t}{w}}\quad \forall f\in \mow{p}{t}{w}. 
\end{equation*}
When $r=1,$ $1<p\le t<\infty$   and $1<q\le \infty,$  the vector-valued inequality follows from extrapolation and the weighted Fefferman--Stein inequality for weighted Lebesgue spaces (see \cite[Theorem 5.3]{MR3538648} for the corresponding extrapolation theorem). The rest of the cases follow from the latter and the fact that $\norm{\abs{f}^s}{\mow{p}{t}{w}}=\norm{f}{\mow{sp}{st}{w}}^s$ for any $0<s<\infty.$  The Nikol'skij representation for $\itlw{[p,t]}{s}{q}{w}$ and $\ibesw{[p,t]}{s}{q}{w}$ with $w\in A_\infty$  (property~\eqref{item:last})  has an analogous statement to that of Theorem~\ref{thm:Nikolskij:weighted} with parameters $0<p\le t<\infty,$ $0<q\le \infty$  and $\lebw{p}{w}$ replaced by $\mow{p}{t}{w}.$ In the setting of $\itlw{[p,t]}{s}{q}{w},$ the convergence of the series is in $\swp$ for any choice of parameters; in the case of $\ibesw{[p,t]}{s}{q}{w},$ the convergence of the series holds in $\swp$ if $q=\infty$ and in $\ibesw{[p,t]}{s}{q}{w}$ otherwise. A similar proof to that of Theorem~\ref{thm:Nikolskij:weighted} applies (see  also Remark~\ref{re:app}).

Finally, we next present a counterpart  of Theorem~\ref{thm:ICM:TL:B} in the context of  $\itlw{[p,t]}{s}{q}{w}.$ 

\begin{theorem}\label{thm:ICM:Morrey}  For $m \in \re,$ let $\sigma(\xi,\eta),$ $\xi,\eta\in\rn,$ be an inhomogeneous Coifman-Meyer multiplier of order $m.$ 
\begin{enumerate}[(a)]
\item If $w\in A_\infty,$ $0<p\le t<\infty,$ $0<p_1\le t_1<\infty$ and $0<p_2\le t_2<\infty$  are  such that $\hcline$ and $1/t=1/t_1+1/t_2,$  $0 < q \le \infty$  and  $s > \tau_{p,q}(w),$  it holds that
\begin{equation*}
\norm{T_\sigma(f,g)}{\itlw{[p,t]}{s}{q}{w}} \lesssim \norm{f}{\itlw{[p_1,t_1]}{s+m}{q}{w}} \norm{g}{h_{p_2}^{t_2}(w)} +  \norm{f}{h_{p_1}^{t_1}(w)}   \norm{g}{\itlw{[p_2,t_2]}{s+m}{q}{w}} \quad \forall f, g \in \sw.
\end{equation*}
Different pairs of $p_1,p_2$ and $t_1, t_2$ can be used on the right-hand side of the inequality above. 
Moreover, if $w\in A_\infty,$ $0<p\le t<\infty,$ $0<q\le \infty$   and $s > \tau_{p,q}(w),$  it holds that
\begin{equation*}
\norm{T_\sigma(f,g)}{\itlw{[p,t]}{s}{q}{w}} \lesssim \norm{f}{\itlw{[p,t]}{s+m}{q}{w}} \norm{g}{L^\infty} +  \norm{f}{L^\infty}   \norm{g}{\itlw{[p,t]}{s+m}{q}{w}} \quad \forall f, g \in \sw.
\end{equation*}
\item If $w_1,w_2\in A_{\infty},$ $w:=w_1^{p/p_1}w_2^{p/p_2},$ $0<p\le t<\infty,$ $0<p_1,p_2<\infty$ are such that $\hcline$ and $s> \tau_{p,q}(w),$ it holds that
\begin{equation*}
\norm{T_\sigma(f,g)}{\itlw{[p,t]}{s}{q}{w}} \lesssim \norm{f}{\itlw{[p_1,p_1t/p]}{s+m}{q}{w_1}} \norm{g}{h_{p_2}^{p_2t/p}(w_2)} +  \norm{f}{h_{p_1}^{p_1t/p}(w_1)}   \norm{g}{\itlw{[p_2,p_2t/p]}{s+m}{q}{w_2}} \quad \forall f, g \in \sw.
\end{equation*}
\end{enumerate}
\end{theorem}
Applying the lifting property $ \norm{f}{\itl{[p,t]}{s}{q}(w)}\simeq\norm{J^s f}{\itl{[p,t]}{0}{q}(w)},$ valid for $s \in \re,$ $0<p\le t<\infty$ and  $0<q\le \infty,$ and under the assumptions  of Theorem~\ref{thm:ICM:Morrey} we obtain, in particular,
\begin{equation}
\norm{J^s(fg)}{\itlw{[p,t]}{0}{q}{w}} \lesssim \norm{J^sf}{\itlw{[p_1,t_1]}{0}{q}{w}} \norm{g}{h_{p_2}^{t_2}(w)} +  \norm{f}{h_{p_1}^{t_1}(w)}   \norm{J^sg}{\itlw{[p_2,t_2]}{0}{q}{w}};\label{eq:ktmorrey}
\end{equation}
\begin{equation*}
\norm{J^s(fg)}{\itlw{[p,t]}{0}{q}{w}} \lesssim \norm{J^sf}{\itlw{[p,t]}{0}{q}{w}} \norm{g}{L^\infty} +  \norm{f}{L^\infty}   \norm{J^sg}{\itlw{[p,t]}{0}{q}{w}};
\end{equation*}
\begin{equation*}
\norm{J^s(fg)}{\itlw{[p,t]}{0}{q}{w}} \lesssim \norm{J^sf}{\itlw{[p_1,p_1t/p]}{0}{q}{w_1}} \norm{g}{h_{p_2}^{p_2t/p}(w_2)} +  \norm{f}{h_{p_1}^{p_1t/p}(w_1)}   \norm{J^sg}{\itlw{[p_2,p_2t/p]}{0}{q}{w_2}}.
\end{equation*}
We refer the reader to \cite[Theorem 6.3]{MR3513582} for unweighted estimates in Morrey spaces in the spirit of  \eqref{eq:ktmorrey}.

\subsection{Leibniz-type rules in the settings of variable-exponent Triebel--Lizorkin and Besov spaces.}
Let  $\P_0$ be the collection of  measurable functions $\pp : \rn \rightarrow (0,\infty)$  such that
\begin{equation*}
p_- := \essinf_{x\in \rn} p(x)>0 \quad\text{ and } \quad p_+ := \esssup_{x\in \rn} p(x)<\infty.
\end{equation*}
For $\pp\in\P_0,$   the variable-exponent Lebesgue space $\Lp$
consists of all measurable functions $f$ such that 
\begin{equation*}
\norm{f}{L^{p(\cdot)}}:=\inf\left\{\lambda>0: \int_{\rn} \abs{\frac{f(x)}{\lambda}}^{p(x)}\,dx\le 1\right\}<\infty;
\end{equation*}
such quasi-norm turns $\Lp$ into  a quasi-Banach  space (Banach space if $p_-\ge1$). We note that if $\pp=p$ is constant then  $\Lp\simeq L^p$ with equality of norms and that
\begin{equation}\label{eq:power}
\norm{\abs{f}^t}{\Lp}=\norm{f}{L^{t\pp}}^t\quad \forall\, t>0.
\end{equation} 
We refer the reader to the books Cruz-Uribe--Fiorenza~\cite{MR3026953} and Diening et al.~\cite{MR2790542} for more information about variable-exponent Lebesgue spaces. 

Let $\B$ be the family of all $\pp \in \P_0$ such that $\M,$  the Hardy--Littlewood maximal operator, is bounded from $\Lp$ to $\Lp.$ A necessary condition for $\pp\in \B$ is  $p_->1;$   sufficient conditions for $\pp \in \B$ include  log-H\"older continuity assumptions.
Property \eqref{eq:power} and Jensen's inequality imply that if $\pp\in \P_0$ and $0<\tau_0<\infty$ is such that $\pp/\tau_0\in\B$ then $\pp/\tau\in\B$ for $0<\tau<\tau_0.$ We then define
\begin{equation*}
\tau_{\pp}=\sup\{\tau>0:\fr{\pp}{\tau}\in \B\},\quad \pp\in \P_0^*,
\end{equation*}
where $\P_0^*$ denotes the class of variable exponents in $\P_0$ such that $\pp/\tau_0\in\B$ for some $\tau_0>0.$ Note that $\tau_{\pp}\le p_-.$

Given $s\in\re,$ $0<q\le \infty$ and $\pp\in\P_0,$ the corresponding inhomogeneous Triebel-Lizorkin and Besov spaces are denoted by $\itl{\pp}{s}{q}$  and  $\ibes{\pp}{s}{q},$
respectively. If $\pp\in\P_0^*,$ these spaces are independent of the functions $\psi$ and $\varphi$ given in  Section~\ref{sec:spaces} (see  Xu~\cite{MR2431378}), contain $\sw$ and are quasi-Banach spaces. If $\pp\in\B$ and $s>0,$ $\itl{\pp}{s}{2}$ coincides with the variable-exponent  Sobolev space $W^{s,\pp}$ (see Gurka et al.~\cite{MR2339558} and Xu~\cite{MR2449626}). More general versions of variable-exponent  Triebel--Lizorkin and Besov spaces, where $s$ and  $q$ are also allowed to be functions, were introduced in  Diening at al.~\cite{MR2498558} and Almeida--H\"ast\"o~\cite{MR2566313}, respectively. 
The local Hardy space with variable exponent $\pp\in \P_0,$  denoted $h^{\pp},$  is defined analogously to $h^p(w)$ with the quasi-norm in $\lebw{p}{w}$ replaced by the quasi-norm in $\Lp.$

We next consider properties \eqref{item:first}-\eqref{item:last} in the variable-exponent setting. Given $\pp\in \P_0,$ $0<q\le\infty$ and $s\in \re,$  property \eqref{item:first} for $\itl{\pp}{s}{q}$ and $\ibes{\pp}{s}{q}$ with $r=\min(p_-,q,1)$ follows right away using \eqref{eq:power}.  Property \eqref{item:second} is given by the following version of H\"older's inequality  in the context of variable-exponent Lebesgue spaces: If $\ppo,\, \ppt,\, \pp \in \P_0$ are such that
$ {1}/{\pp} = {1}/{\ppo} + {1}/{\ppt}$ then
\begin{equation*}
 \|fg\|_{L^\pp} \lesssim \|f\|_{L^\ppo}\|g\|_{L^\ppt}\quad \forall f\in L^\ppo, g\in L^\ppt.
 \end{equation*}
 For a proof with exponents in $\P_0$ such that $p_-\ge1$ see, for instance,  \cite[Corollary 2.28]{MR3026953}; the general case with exponents in $\P_0$ follows from the latter and \eqref{eq:power}. Regarding property \eqref{item:third} for  variable-exponent Lebesgue spaces, the following version of the Fefferman-Stein inequality follows from \cite[Section 5.6.8]{MR3026953} and \eqref{eq:power}: If $\pp\in\P_0^*,$  $0<q\le \infty$  and $0<r<\min(\tau_{\pp},q)$ then
\begin{equation*}
\norm{\left(\sum_{j\in\ent} \abs{\M_r(f_j)}^q\right)^{\frac{1}{q}}}{\Lp}\lesssim \norm{\left(\sum_{j\in\na_0} \abs{f_j}^q\right)^{\frac{1}{q}}}{\Lp}\quad \forall \{f_j\}_{j\in \na_0}\in \Lp(\ell^q);
\end{equation*}
in particular, if $0<r<\tau_{\pp}$ it holds that 
\begin{equation*}
\norm{\M_r(f)}{\Lp}\lesssim \norm{f}{\Lp}\quad \forall f\in \Lp.
\end{equation*}
Finally, the following version of the Nikol'skij representation for $\itl{\pp}{s}{q}$ and $\ibes{\pp}{s}{q}$ (property \eqref{item:last}), can be proved along the lines of the proof of Theorem~\ref{thm:Nikolskij:weighted} (see also Remark~\ref{re:app}):
\begin{theorem}\label{thm:Nikolskij:variable} For $\A> 0,$ let $\{u_j\}_{j \in \ent} \subset \mathcal{S}'(\rn)$ be a sequence of tempered distributions such that $\supp(\widehat{u_j}) \subset B(0, \A\, 2^j ) $ for all $j \in \ent.$
Let $ \pp\in \P_0^*$, $0 < q \le \infty$ and $s > n(1/\min(\tau_{\pp},q,1)-1)$. If $\norm{\{2^{js} u_j\}_{j\in\ent}}{\Lp(\ell^{q})} < \infty$, then the series $\sum_{j \in \ent} u_j$ converges in $\itl{\pp}{s}{q}$ (in $\swp$ if $q=\infty$) and 
\begin{equation*}
\norm{\sum_{j \in \na_0} u_j}{\itl{\pp}{s}{q}} \lesssim  \norm{\{2^{js} u_j\}_{j\in\na_0}}{L^\pp(\ell^{q})},
\end{equation*}
where the implicit constant depends only on $n,$ $\A,$ $s,$ $\pp$ and  $q.$  An analogous statement holds true for $\ibes{\pp}{s}{q}$ with $s> n(1/\min(\tau_{\pp},1)-1)$
\end{theorem}

We next state a version of Theorem~\ref{thm:ICM:TL:B} for variable-exponent Triebel-Lizorkin spaces as a model result.
\begin{theorem} \label{thm:ICM:variable} For $m \in \re,$ let $\sigma(\xi,\eta),$ $\xi,\eta\in\rn,$ be an inhomogeneous Coifman-Meyer multiplier of order $m.$ If  $\pp,\ppo,\ppt\in \P_0^*$  are such that $ {1}/{\pp} = {1}/{\ppo} + {1}/{\ppt},$   $0 < q \le \infty$ and  
$s > n(1/\min(\tau_{\pp},q,1)-1),$  it holds that
\begin{equation*}
\norm{T_\sigma(f,g)}{\itl{\pp}{s}{q}} \lesssim \norm{f}{\itl{\ppo}{s+m}{q}} \norm{g}{h^{\ppt}} +  \norm{f}{h^{\ppo}}   \norm{g}{\itl{\ppt}{s+m}{q}} \quad \forall f, g \in \sw.
\end{equation*}
Moreover, if $\pp\in \P_0^*,$     $0 < q \le \infty$ and  
$s > n(1/\min(\tau_{\pp},q,1)-1),$  it holds that
\begin{equation*}
\norm{T_\sigma(f,g)}{\itl{\pp}{s}{q}} \lesssim \norm{f}{\itl{\pp}{s+m}{q}} \norm{g}{L^\infty} +  \norm{f}{L^\infty}   \norm{g}{\itl{\pp}{s+m}{q}} \quad \forall f, g \in \sw.
\end{equation*}
\end{theorem}

The lifting property $ \norm{f}{\itl{\pp}{s}{q}}\simeq\norm{J^s f}{\itl{\pp}{0}{q}}$ holds true for $s \in \re,$  $\pp\in \P_0^*$  and $0<q\le\infty;$ then, under the assumptions of Theorem~\ref{thm:ICM:variable}  we obtain, in particular,
\begin{equation*}
\norm{J^s (fg)}{\itl{\pp}{0}{q}} \lesssim \norm{J^s f}{\itl{\ppo}{0}{q}} \norm{g}{h^{\ppt}} +  \norm{f}{h^{\ppo}}   \norm{J^sg}{\itl{\ppt}{0}{q}};
\end{equation*}
\begin{equation*}
\norm{J^s(fg)}{\itl{\pp}{0}{q}} \lesssim \norm{J^sf}{\itl{\pp}{0}{q}} \norm{g}{L^\infty} +  \norm{f}{L^\infty}   \norm{J^sg}{\itl{\pp}{0}{q}}.
\end{equation*}
These last two estimates  extend some of the inequalities in \cite[Theorem 1.2]{MR3513582}, where Leibniz-type rules for the product of two functions were proved in variable-exponent Lebesgue spaces through the use of extrapolation techniques.

\appendix

\section{}\label{sec:appendix}

In this section we briefly sketch the proof of Theorem \ref{thm:Nikolskij:weighted} which follows the same ideas of an unweighted version for inhomogeneous spaces in \cite[Theorem 3.7]{MR837335}. We start by presenting useful lemmas and inequalities  and then proceed with the proof. 
\begin{lemma}[Particular case of Corollary 2.11 in \cite{MR837335}]\label{coro:2:11}
Suppose $0 < r \leq 1,$  $A >0$, $R \geq 1$ and $\ex > n/r$. If $\phi \in \sw$ and $f$ is such that  $\supp(\widehat{f}) \subset \{\xi\in\rn:\abs{\xi}\le AR\}$, it holds that
\begin{equation*}
|\phi * f(x)| \lesssim R^{n (\frac{1}{r}  -1)} A^{-n} \norm{(1 + |A \cdot|)^\ex \phi}{L^\infty}  \M_rf(x) \quad \forall x \in \rn,
\end{equation*}
where the implicit constant is independent of $A, R, \phi,$ and $f.$  
\end{lemma}
\begin{remark} \cite[Corollary 2.11]{MR837335} incorrectly states $A^{-n/r}$ instead of $ A^{-n}$. Also, it states $A \geq 1$, but the result is true for $A >0$ as stated in Lemma~\ref{coro:2:11}.
\end{remark}

The following lemma is a weighted version of   \cite[Corollary 2.12 (1)]{MR837335}. We include its brief proof for completeness.
\begin{lemma}\label{coro:2:12(1)}
Suppose $w\in A_\infty,$ $0 < p \leq \infty$, $A > 0$, $R \geq 1,$  and $\ex > b > n/ \min(1,p/\tau_w).$  If $\phi \in \sw$ and $f$ is such that $\supp(\widehat{f}) \subset \{\xi\in\rn:\abs{\xi}\le AR\}$, it holds that
\begin{equation*}
\norm{\phi * f}{\lebw{p}{w}} \lesssim R^{b - n} A^{-n} \norm{(1 + |A \cdot|^\ex) \phi}{L^\infty} \norm{f}{\lebw{p}{w}},
\end{equation*}
where the implicit constant is independent of $A,$ $R,$ $\phi$ and $f.$ 
\end{lemma}

\begin{proof}  Set $r:=n/b < \min(1, p/\tau_w).$ The hypothesis $\ex > b$  means $ \ex  > n/r$ and Lemma~\ref{coro:2:11}  yields
$$
|\phi * f(x)| \lesssim R^{n (\frac{1}{r} -1)} A^{-n} \norm{(1 + |A \cdot|)^\ex \phi}{L^\infty}  \M_rf(x) \quad \forall x \in \rn.
$$
Since $r < p/\tau_w,$ we have   $\norm{\M_rf}{\lebw{p}{w}} \lesssim \norm{f}{\lebw{p}{w}}$ and therefore
$$
\norm{\phi * f}{\lebw{p}{w}} \lesssim R^{n (\frac{1}{r} -1)} A^{-n} \norm{(1 + |A \cdot|)^\ex \phi}{L^\infty} \norm{f}{\lebw{p}{w}};
$$
observing that $1/r -1 = (b -n)/n$, the desired estimate follows. 
\end{proof}
The following lemma is a modified version of \cite[Lemma 3.8]{MR837335}.
\begin{lemma}\label{eq:seriesineq} Let  $\tau < 0$, $\lambda \in \re$, $0 < q \leq \infty$, and $k_0 \in \ent$. Then, for any sequence $\{d_j\}_{j \in \ent} \subset [0, \infty)$ it holds that
\begin{equation*}
\norm{\left\{\sum_{k=k_0}^\infty 2^{\tau k} 2^{\lambda(j+k)}d_{j+k}\right\}_{j\in\ent}}{\ell^q} \lesssim \norm{\{2^{j \lambda} d_j\}_{j\in\ent}}{\ell^q},
\end{equation*}
where the implicit constant depends only on $k_0, \tau, \lambda$ and $q$.
\end{lemma}
\begin{proof}
Suppose first that $0 < q \leq 1$. Then,
\begin{align*}
& \norm{\left\{\sum_{k=k_0}^\infty 2^{\tau k} 2^{\lambda(j+k)}d_{j+k}\right\}_{j\in \ent}}{\ell^q} = \left[\sum_{j \in \ent} \left( \sum_{k=k_0}^\infty 2^{\tau k} 2^{\lambda(j+k)}d_{j+k} \right)^q \right]^\frac{1}{q}\\
& \quad\quad\quad\quad\quad \leq  \left[\sum_{j \in \ent}  \sum_{k=k_0}^\infty 2^{\tau q k} 2^{\lambda q (j+k)}d_{j+k}^q \right]^\frac{1}{q} =  \left[ \sum_{k=k_0}^\infty  2^{\tau q k} \sum_{j \in \ent} 2^{\lambda q (j+k)}d_{j+k}^q \right]^\frac{1}{q}\\
&  \quad\quad\quad \quad \quad = \left(  \sum_{k=k_0}^\infty   2^{\tau q k}     \right)^\frac{1}{q} \norm{\{2^{j \lambda} d_j\}_{j\in \ent}}{\ell^q} = C_{k_0, \tau, q} \norm{\{2^{j \lambda} d_j\}_{j\in\ent}}{\ell^q},
\end{align*}
where in the last equality we have used that $\tau < 0$. If $1 < q < \infty$ we use H\"older's inequality with $q$ and $q'$ to write
\begin{align*}
 \norm{\left\{\sum_{k=k_0}^\infty 2^{\tau k} 2^{\lambda(j+k)}d_{j+k}\right\}_{j\in\ent}}{\ell^q} 
& \le \left[\sum_{j \in \ent}   \left(\sum_{k=k_0}^\infty 2^{\tau k q/2} 2^{\lambda q(j+k)}d_{j+k}^q \right)     \left( \sum_{k=k_0}^\infty 2^{\tau kq'/2} \right)^{q/q'}   \right]^\frac{1}{q}\\
&    =C_{k_0, \tau, q} \norm{\{2^{j \lambda} d_j\}_{j\in\ent}}{\ell^q}.
\end{align*}

The case $q = \infty$ is straightforward.
\end{proof}

\begin{proof}[Proof of  Theorem~\ref{thm:Nikolskij:weighted}] 
We first prove Theorem~\ref{thm:Nikolskij:weighted} for finite families. We will do this in the homogeneous settings, with the proof in the inhomogeneous settings being similar. Suppose $\{u_j\}_{j\in\ent}$ is such that $u_j=0$ for all $j$ except those belonging to some finite subset of $\ent;$  this assumption allows us to avoid convergence issues since all the sums  considered will be finite.

For Part~\eqref{item:thh:Nikolskij:TL}, let $\A,$ $w,$ $p,$ $q$ and $s$ be as in the hypotheses. Fix $0<r<\min(1,p/\tau_w,q)$ such that  $s>n(1/r-1);$ note that the latter is possible since $s>\tau_{p,q}(w).$

Let $k_0 \in \ent$ be such that $2^{k_0 -1} < \A\leq 2^{k_0},$ then  
$$
\supp(\widehat{u_\ell}) \subset B(0, 2^\ell \A) \subset B(0, 2^{\ell +k_0}) \quad \forall \ell \in \ent.
$$
Define $u =\sum_{\ell \in \ent} u_\ell$ and let $\psi$ be as in the definition of $\tlw{p}{s}{q}{w}$ in Section~\ref{sec:spaces}. We have
\begin{equation}\label{psi*u}
\Delta^\psi_j u = \sum_{\ell \in \ent} \Delta^\psi_j u_\ell = \sum_{\ell = j - k_0}^\infty \Delta^\psi_j  u_\ell = \sum_{k = - k_0}^\infty \Delta^\psi_j u_{j+k}.
\end{equation}
We will use Lemma~\ref{coro:2:11} with $\phi(x) =  2^{jn} \psi(2^j x),$  $f = u_{j+k},$  $A= 2^{j  } >0,$ and $R= 2^{k + k_0}.$ (Notice that $\supp(\widehat{u_{j+k}}) \subset  B(0, 2^{j} 2^{k + k_0})$ and, since $k \geq -k_0$, we get $R \geq 1$.) Fixing  $\ex>n/r$ and applying Lemma~\ref{coro:2:11}, we get
\begin{align*}
|\Delta^\psi_j  u_{j+k}(x)| & \lesssim   2^{k_0n(\frac{1}{r} -1)} 2^{k n (\frac{1}{r} -1)}  2^{-jn} \norm{(1 + |2^j \cdot|)^\ex 2^{jn}\psi (2^j\cdot)}{L^\infty}  \M_r(u_{j+k})(x)\\
& \sim  2^{k n (\frac{1}{r} -1)} \left( \sup_{y \in \rn} (1 + |2^j y|)^\ex |\psi(2^j y)|\right) \M_r(u_{j+k})(x).
\end{align*}
 Hence,
\begin{align*}
2^{js}|\Delta^\psi_j  u_{j+k}(x)| & \lesssim 2^{k n (\frac{1}{r} -1-\frac{s}{n})} 2^{s(j+k)}\M_r(u_{j+k})(x),
\end{align*}
and then, recalling \eqref{psi*u}, 
\begin{align*}
2^{js} |\Delta^\psi_j  u(x)| & \lesssim  \sum_{k =  - k_0}^\infty 2^{k n (\frac{1}{r} -1-\frac{s}{n})} 2^{s(j+k)}\M_r(u_{j+k})(x).
\end{align*}
Since $1/r -1-s/n < 0$,   Lemma~\ref{eq:seriesineq}  yields
$$
\norm{\{ 2^{js} |\Delta^\psi_j  u|\}_{j\in\ent}}{L^p(w)(\ell^{q})} \lesssim \norm{\{2^{js} \M_ru_j\}_{j\in\ent}}{L^p(w)(\ell^{q})}
$$
with an implicit constant independent of $\{u_j\}_{j\in \ent}.$
Applying the weighted  Fefferman-Stein inequality  to the right-hand side of the last inequality leads to the desired estimate
\begin{equation*}
\norm{u}{\tlw{p}{s}{q}{w}} \lesssim \norm{\{2^{js}u_j\}_{j\in\ent}}{L^p(\ell^{q})}.
\end{equation*}

For Part~\eqref{item:thh:Nikolskij:B}, let $\A,$ $w,$ $p,$ $q$ and $s$ be as in the hypotheses and $k_0$ be as above. 
Consider $\Delta^\psi_j  u_{j+k}$  in  \eqref{psi*u} and  apply  Lemma~\ref{coro:2:12(1)}  with $\phi(x)=2^{jn}\psi(2^{-j} x)$, $f= u_{j+k},$ $A=2^{j},$ $R=2^{k +  k_0}$,  $\ex>b$ and $  n /\min(1,p/\tau_w)<b<n+s;$ note that such $b$ exists since $s>\tau_p(w).$ We get
\begin{align*}
 \norm{\Delta^\psi_j  u_{j+k}}{L^p(w)}  \lesssim 2^{(k + k_0)(b-n)} 2^{- j  n} \norm{(1 + |2^j \cdot|)^\ex 2^{jn}\psi(2^{-j}\cdot)}{L^\infty} \norm{u_{j+k}}{L^p(w)}\sim 2^{k(b-n)}   \norm{u_{j+k}}{L^p(w)},
\end{align*}
and setting $p^*:=\min(p,1)$ we obtain
\begin{align*}
2^{js p^*}\norm{\Delta^\psi_j  u}{L^p(w)}^{p^*} \lesssim 2^{js p^*}\sum_{k = - k_0}^\infty \norm{\Delta^\psi_j  u_{j+k}}{L^p(w)}^{p^*}
 =  \sum_{k = - k_0}^\infty 2^{k(b-n -s) p^*} 2^{s p^* (j+k)}   \norm{u_{j+k}}{L^p(w)}^{p^*}.
\end{align*}
Hence, applying Lemma~\ref{eq:seriesineq}, it follows that
\begin{align*}
\norm{u}{\besw{p}{s}{q}{w}}  \lesssim  \norm{\left\{ \sum_{k = - k_0}^\infty 2^{k(b-n -s) p^*}  2^{s p^* (j+k)}   \norm{u_{j+k}}{L^p(w)}^{p^*}\right\}_{j\in\ent}}{\ell^{q/p*}}^\frac{1}{p^*}\lesssim  \norm{\{2^{js} u_j\}_{j\in\ent}}{\ell^{q}(L^p(w))}, 
\end{align*}
as desired.

We next show the theorem for any not necessarily finite family. For ease of notation, we only work in the context of weighted homogeneous Triebel--Lizorkin spaces; the reasoning is identical for the other settings.  Let $\{u_j\}_{j\in \ent},$ $w,$ $p,$ $q,$ and $s$ be as in the hypotheses. Define $U_N:=\sum_{k=-N}^N u_j;$  since the theorem is true for finite families and, for $M<N,$ $\{u_j\}_{M+1\le \abs{j}\le N}$ satisfies the hypotheses of the theorem, we have
\begin{equation}\label{eq:convergence}
\norm{U_N-U_M}{\tlw{p}{s}{q}{w}}\lesssim \norm{\{2^{js} u_j\}_{M+1\le \abs{j}\le N}}{\lebw{p}{w}(\ell^q)},
\end{equation}
where the implicit constant is independent of $M,$ $N$ and the family $\{u_j\}_{j\in\ent}.$ 

If $0<q<\infty,$ as $M,N\to\infty,$  the right-hand side of \eqref{eq:convergence} tends to zero by the assumption $\norm{\{u_j\}_{j\in \ent}}{\lebw{p}{w}(\ell^q)} <\infty$ and the dominated convergence theorem; therefore, since $\tlw{p}{s}{q}{w}$ is complete, $\sum_{j\in\ent} u_j$ converges in $\tlw{p}{s}{q}{w}.$ The same reasoning used to obtain \eqref{eq:convergence} gives that
\[
\norm{U_N}{\tlw{p}{s}{q}{w}}\lesssim \norm{\{2^{js} u_j\}_{-N\le j\le N}}{\lebw{p}{w}(\ell^q)},
\]
where the implicit constant is independent of $N$ and the family $\{u_j\}_{j\in\ent}.$ It then  follows that
\[
\norm{\sum_{j\in\ent} u_j}{\tlw{p}{s}{q}{w}}\lesssim \norm{\{2^{js} u_j\}_{j\in\ent}}{\lebw{p}{w}(\ell^q)},
\]
with the implicit constant  independent of  the family $\{u_j\}_{j\in\ent}.$

If $q=\infty,$  we use that $\{2^{(s-\varepsilon)j}u_j\}_{j\ge 0}$ and $\{2^{(s+\varepsilon)j}u_j\}_{j< 0}$ belong to $\ell^1(L^p(w))$ for any $\varepsilon>0$ and apply Theorem~\ref{thm:Nikolskij:weighted} under the case of finite $q$ to conclude that $\sum_{j=0}^Nu_j$ and $\sum_{j=-N}^{-1}u_j$ converge in $\besw{p}{s-\varepsilon}{1}{w}$ and $\besw{p}{s+\varepsilon}{1}{w},$ respectively (choosing $\varepsilon>0$ so that $s-\varepsilon>\tau_{p,q}(w)\ge \tau_p(w)$). Therefore, $U_N$ convergence in $\mathcal{S}'_0(\rn).$ Moreover, by  Theorem~\ref{thm:Nikolskij:weighted} applied to the finite sequence $\{u_j\}_{-N\le j\le N},$ we have that $U_N\in \tlw{p}{s}{\infty}{w}$ and 
\[
\norm{U_N}{\tlw{p}{s}{\infty}{w}}\lesssim \norm{\{2^{js} u_j\}_{-N\le j\le N}}{\lebw{p}{w}(\ell^\infty)}\le \norm{\{2^{js} u_j\}_{j\in \ent}}{\lebw{p}{w}(\ell^\infty)},
\]
with the implicit constant independent of $N$ and $\{u_j\}_{j\in\ent}.$
Since $\tlw{p}{s}{\infty}{w}$ has the Fatou property (see Remark~\ref{re:app}), we conclude that $\lim_{N\to \infty} U_N=\sum_{j\in\ent}u_j$ belongs to $\tlw{p}{s}{\infty}{w}$ and 
\[
\norm{\sum_{j\in\ent}u_j}{\tlw{p}{s}{\infty}{w}}\lesssim  \norm{\{2^{js} u_j\}_{j\in \ent}}{\lebw{p}{w}(\ell^\infty)}.
\]
\end{proof}

\begin{remark}\label{re:app} As stated in Section~\ref{sec:more}, a Nikol'skij representation theorem holds true for Triebel--Lizorkin and Besov spaces based on weighted Lorentz spaces, weighted Morrey spaces and variable Lebesgue spaces. We next make some remarks concerning the proofs of the corresponding versions of Theorem~\ref{thm:Nikolskij:weighted} in such settings:
\begin{enumerate}[(a)] 

\item Regarding the proof of Part~\eqref{item:thh:Nikolskij:TL} of Theorem~\ref{thm:Nikolskij:weighted} (for instance, in the inhomogeneous case) the fact that $\norm{\{2^{js} u_j\}_{M+1\le \abs{j}\le N}}{\lebw{p}{w}(\ell^q)}$ converges to zero, as $M,N\to\infty,$  when $q$ is finite, allows to conclude that $\sum_{j\in\naz} u_j$ converges in $\itlw{p}{s}{q}{w}$ through the use of \eqref{eq:convergence}. Under the hypothesis of  Part~\eqref{item:thh:Nikolskij:TL} for $\mathcal{X}=\lebw{p,t}{w}$ with $0<p,t<\infty$ or $\mathcal{X}=\Lp$ with $\pp\in \P_0$ and $q$  finite,  it holds that 
\begin{equation}\label{eq:cauchyseq}
\norm{\{2^{js} u_j\}_{M+1\le \abs{j}\le N}}{\mathcal{X}(\ell^q)}\to 0\quad \text{as } M,N\to\infty;
\end{equation}
therefore, $\sum_{j\in\naz} u_j$ converges in $\itlw{(p,t)}{s}{q}{w}$ and  $\itl{\pp}{s}{q},$  respectively.   The fact \eqref{eq:cauchyseq} is a consequence of a dominated convergence type theorem in $\mathcal{X}$ and the corresponding assumptions in Part~\eqref{item:thh:Nikolskij:TL}. For the indices for which \eqref{eq:cauchyseq} does not necessarily hold under the corresponding assumptions in Part~\eqref{item:thh:Nikolskij:TL} ($t=\infty$ or $q=\infty$ when  $\mathcal{X}=\lebw{p,t}{w},$ $0<p\le t<\infty$ and $0<q\le \infty$ when  $\mathcal{X}=\mow{p}{t}{w},$ $q=\infty$ when  $\mathcal{X}=\Lp$), the convergence of $\sum_{j\in\naz} u_j$ holds in $\swp$ rather than in  $\itlw{(p,t)}{s}{q}{w},$ $\itlw{[p,t]}{s}{q}{w}$ or  $\itl{\pp}{s}{q},$ respectively. Regarding Part~\eqref{item:thh:Nikolskij:B}, the counterpart of \eqref{eq:cauchyseq} is
\begin{equation*}
\norm{\{2^{js} u_j\}_{M+1\le \abs{j}\le N}}{\ell^q(\mathcal{X})}\to 0\quad \text{as }M,N\to\infty,
\end{equation*}
which is always true under the corresponding assumptions of Part~\eqref{item:thh:Nikolskij:B} as long as $q$ is finite, in which case the convergence of $\sum_{j\in\naz} u_j$ holds in the corresponding $\mathcal{X}$-based Besov space. If $q=\infty,$ the convergence is in $\swp$ rather than in the $\mathcal{X}$-based Besov space.

\item  The last part of the proof of Theorem~\ref{thm:Nikolskij:weighted} uses the Fatou property of Triebel--Lizorkin and Besov spaces.  Let  $\mathcal{A}$ be a quasi-Banach space such that $\sw\hookrightarrow\mathcal{A}\hookrightarrow\swp$ (or $\swz\hookrightarrow\mathcal{A}\hookrightarrow\mathcal{S}_0'(\rn)$). The space $\mathcal{A}$ is said to  have the Fatou property if  for every sequence $\{f_j\}_{j\in \na}\subset \mathcal{A}$ that converges in $\swp$ ($\mathcal{S}_0'(\rn),$  respectively), as $j\to \infty,$ and that satisfies $\liminf_{j\to\infty} \norm{f_j}{\mathcal{A}}<\infty,$ it follows that $\lim_{j\to\infty}f_j\in \mathcal{A}$ and $\norm{\lim_{j\to\infty}f_j}{\mathcal{A}}\lesssim \liminf_{j\to\infty} \norm{f_j}{\mathcal{A}},$ where the implicit constant is independent of $\{f_j\}_{j\in\na}.$ 

It can be shown, using standard proofs, that Triebel--Lizorkin and Besov spaces based on a quasi-Banach space $\mathcal{X}$ of measurable functions (i.e. $\itl{\mathcal{X}}{s}{q},$ $\ibes{\mathcal{X}}{s}{q}$ and their homogeneous counterparts) posses the Fatou property for any $s\in \re$ and $0<q\le \infty$ if $\mathcal{X}$ satisfies the following properties: (1) if $f,g\in \mathcal{X}$ and $\abs{f}\le\abs{g}$ pointwise a.e., then $\norm{f}{\mathcal{X}}\lesssim \norm{g}{\mathcal{X}};$ (2) if $\{f_j\}_{j\in \na}\subset \mathcal{X}$ and $f_j\ge 0$ poinwise a.e., then $\norm{\liminf_{j\to\infty} f_j}{\mathcal{X}}\lesssim \liminf_{j\to\infty}\norm{f_j}{\mathcal{X}}.$ Given a weight $w,$  properties (1) and (2) are easily verified  for $L^p(w)$ if $0<p\le \infty,$ $L^{p,t}(w)$ if $0<p<\infty,$ $0<t\le \infty,$ $M^{t}_p(w)$ if $0<p\le t<\infty;$ they also hold for $L^{\pp}$ if $\pp\in \P_0,$ as shown in \cite[Theorem 2.61]{MR3026953}. As a consequence, all the Triebel--Lizorkin and Besov spaces considered in the statements of the theorems in Sections~\ref{sec:results} and \ref{sec:more} have the Fatou--Property.

\end{enumerate}
\end{remark}

\def\ocirc#1{\ifmmode\setbox0=\hbox{$#1$}\dimen0=\ht0 \advance\dimen0
  by1pt\rlap{\hbox to\wd0{\hss\raise\dimen0
  \hbox{\hskip.2em$\scriptscriptstyle\circ$}\hss}}#1\else {\accent"17 #1}\fi}

\end{document}